\documentclass[a4paper,11pt]{amsart}
\usepackage{amssymb}
\usepackage{mathrsfs}
\usepackage{amsmath,amssymb,amsthm,latexsym,amscd,mathrsfs}
\usepackage{indentfirst}
\usepackage{stmaryrd}
\usepackage{graphicx}
\usepackage{extarrows}
\usepackage{romanbar}
\usepackage{multirow}
\usepackage{latexsym}
\usepackage{amsfonts}
\usepackage{color}
\usepackage{pictexwd,dcpic}
\usepackage{graphicx}
\usepackage{psfrag}
\usepackage{hyperref}
\usepackage{comment}
\usepackage{enumerate}
\usepackage{cite}%new
\allowdisplaybreaks

\numberwithin{equation}{section} \theoremstyle{plain}
\newtheorem{thm}{Theorem}[section]

\newtheorem{prop}[thm]{Proposition}
\newtheorem{asser}[thm]{Assertion}
\newtheorem{lem}[thm]{Lemma}

\newtheorem{rem}[thm]{Remark}

\newtheorem{ack}{Acknowledgements}   
 \makeatletter

\def\<{\langle}
\def\>{\rangle}
\def\({\left(}
\def\){\right)}
\def\[{\left[}
\def\]{\right]}
\def\tr{\mathop{\text{tr}}}
\renewcommand\subsection{\@startsection{subsection}{2}{\z@}%
  {2.0ex plus 0.8ex minus 0.2ex}% beforeskip
  {1.0ex plus 0.2ex}% afterskip (positive => text starts on next line)
  {\normalfont\bfseries}}

\makeatother
\title{Rigidity of closed minimal hypersurfaces in $\mathbb{S}^5$}

\author[J. Q. Ge]{Jianquan Ge}
\address{School of Mathematical Sciences, Laboratory of Mathematics and Complex Systems, Beijing Normal University, Beijing 100875, P.R. CHINA.}
\email{jqge@bnu.edu.cn}

\author[T. Liu]{Tong Liu$^{\dag}$}
\address{School of Mathematical Sciences, Laboratory of Mathematics and Complex Systems, Beijing Normal University, Beijing 100875, P.R. CHINA.}
\email{tt@mail.bnu.edu.cn}

\author[K. Y. Luo]{Keyan Luo}
\address{School of Mathematical Sciences, Laboratory of Mathematics and Complex Systems, Beijing Normal University, Beijing 100875, P.R. CHINA.}
\email{kyluo@mail.bnu.edu.cn}

\author[W. J. Yan]{Wenjiao Yan}
\address{School of Mathematical Sciences, Laboratory of Mathematics and Complex Systems, Beijing Normal University, Beijing 100875, P.R. CHINA.}
\email{wjyan@bnu.edu.cn}
\makeatletter
\@namedef{subjclassname@2020}{\textup{2020} Mathematics Subject Classification}
\makeatother

\date{}
\thanks {$^{\dag}$ the corresponding author}
\thanks{The project is partially supported by the NSFC (No. 12271038, 12571049, 12526205) and the Fundamental Research Funds for the Central Universities.}
\begin{document}

\begin{abstract}

	The celebrated Chern conjecture asserts that any closed minimal hypersurface in $\mathbb{S}^{n+1}$ with constant scalar curvature is isoparametric. In this paper, we resolve this conjecture in the affirmative for $M^4 \subset \mathbb S^5$ under the assumption that the Gauss-Kronecker curvature $K$ is constant.

This result breaks the traditional reliance on consecutive trace conditions, demonstrating that the nonconsecutive spectral invariant set $\{H, S, K\}$ is sufficient to yield complete geometric rigidity. To overcome the analytical singular locus, we construct two novel weighted $3$-forms adapted to $S$ and $K$. Crucially, the global curvature estimates required to close our analysis are obtained unconditionally by proving the Euler characteristic $\chi(M)=0$. This local-to-global approach provides a new paradigm for higher-dimensional rigidity problems.

%We prove that if $M$ is a closed minimal hypersurface immersed in $\mathbb{S}^5$ with constant squared norm of second fundamental form $S$ and constant Gauss-Kronecker curvature $K$, then $M$ is isoparametric.
\end{abstract}

\maketitle

% \tableofcontents

\section{Introduction}

The Chern conjecture is one of the most celebrated and long-standing rigidity problems in the geometry of submanifolds, specifically concerning minimal hypersurfaces in spheres. Let $M^n$ be a closed minimal hypersurface immersed in the unit sphere $\mathbb S^{n+1}$, and let $S$ denote the squared norm of its second fundamental form. The classical Simons' identity~\cite{sim68} yields the fundamental integral inequality
$$
\int_M S(S-n)\,\mathrm{vol}\ge0.
$$
Consequently, if $S\le n$, then either $S\equiv0$ (a totally geodesic sphere) or $S\equiv n$. The critical threshold $S\equiv n$ was characterized independently by Chern, do Carmo, Kobayashi \cite{cdk70}, and Lawson \cite{law69}, who showed that $M$ must be a Clifford torus. Let $R_{M}$ denote the scalar curvature of $M$. By the Gauss equation, the minimality of $M$ implies that $S=n(n-1)-R_{M}$. %Thus, the constancy of $S$ is equivalent to the constancy of $R_M$. 
Then these pioneering results led to the original Chern Conjecture that the set of values of constant $R_M$ on closed minimal hypersurfaces $M$ is discrete.

The first major breakthrough was achieved by Peng and Terng \cite{pt83}, who proved that if $S>n$, then $S \ge n+\frac{1}{12n}$. This work initiated the systematic study of the higher pinching problem (see recent progress on the second gap in \cite{gtyz25} and references therein). Despite substantial subsequent progress over the past several decades, all currently known closed minimal hypersurfaces with constant $S$ are isoparametric---those with constant principal curvatures. The classification of isoparametric hypersurfaces in unit spheres was initiated by Élie Cartan in the 1930s, and the final case was settled by Quo-Shin Chi \cite{chi20} in 2020. For the classifications and developments of isoparametric hypersurfaces, see the excellent monograph of Cecil and Ryan \cite{cr15} and recent survey \cite{gqty25}. This empirical evidence supports the following stronger, and now standard, formulation of the conjecture:

\vspace{2mm}

\noindent
\textbf{Chern's Conjecture (strong version)}. \emph{Let $M^n$ be a closed minimal hypersurface immersed in $\mathbb S^{n+1}$. If $S$ is constant, then $M$ is isoparametric.}
\vspace{2mm}

%\begin{conj}\label{conj:Chern}
%	Let $M^n$ be a closed minimal hypersurface immersed in $\mathbb S^{n+1}$. If $S$ is constant, then $M$ is isoparametric.
%\end{conj}

To date, the Chern conjecture has been completely resolved only in dimension $n=3$, a landmark achievement due to Chang \cite{changjdg93}. More generally, in dimension $n=3$, de Almeida and
Brito~\cite{ab90}
introduced an elegant differential-form method and proved the conjecture under the
additional assumption $R_M\ge0$; Chang subsequently removed this sign
restriction~\cite{chang93} to establish the following definitive result.
\vspace{2mm}

\noindent
\textbf{Theorem} (de Almeida-Brito-Chang). \emph{Let $M^3$ be a closed hypersurface immersed in $\mathbb S^4$ with constant mean curvature and constant scalar curvature. Then $M$ is isoparametric.}
%\vspace{2mm}

%\begin{thm}[de Almeida-Brito-Chang]\label{thm:ABC}
%	Let $M^3$ be a closed hypersurface in $\mathbb S^4$ with constant mean curvature and constant scalar curvature. Then $M$ is isoparametric.
%\end{thm}

The differential-form framework was generalized partially to low dimensions (see for example \cite{lss05,svw12,dgw17}) and partially to arbitrary dimension \cite{twy20}, culminating in the work of Tang and Yan \cite{ty23}, who  completely generalized it to arbitrary dimensions. Let $\lambda_1,\ldots,\lambda_n$ denote the principal curvatures of $M$, and define the power sums (or traces) of the shape operator $A$ by
$$
P_r=\sum_{i=1}^n\lambda_i^r=\operatorname{tr}A^r.
$$
The Tang-Yan rigidity theorem is formulated as follows.

\vspace{2mm}

\noindent
\textbf{Theorem} (Tang-Yan). \emph{Let $M^n$ $(n>3)$ be a closed hypersurface immersed in $\mathbb S^{n+1}$. If $R_M\ge0$ and the power sums $P_1, P_2, \ldots, P_{n-1}$ are constant, then $M$ is isoparametric. Moreover, if $M^n$ has $n$ distinct principal curvatures somewhere, then $R_M\equiv0$.}
\vspace{2mm}

%\begin{thm}[Tang-Yan]\label{thm:Tang-Yan}
%	Let $M^n$ $(n>3)$ be a closed hypersurface immersed in $\mathbb S^{n+1}$. If $R_M\ge0$ and the power sums $P_1, P_2, \ldots, P_{n-1}$ are constant, then $M$ is isoparametric.
%\end{thm}

This theorem suggests a broader set of questions. Is the nonnegativity of the scalar curvature essential, or can it be removed in arbitrary dimensions?
Must the prescribed power sums be consecutive, or can they be replaced by
another suitably chosen collection? More fundamentally, can the number of
prescribed power sums be reduced?

%We note that the assumption $R_M \geq 0$ has been successfully removed in only a few select scenarios: namely, in dimension $3$ by Chang \cite{chang93}, and in dimension $4$ by Deng, Gu, and Wei \cite{dgw17} under the additional structural assumption that $M^4$ is a Willmore hypersurface. 

For the first open dimension, $M^4\subset\mathbb S^5$, Cheng and Li~\cite{cl25} recently bypassed the sign condition $R_M \geq 0$ by introducing suitable weights into the differential form, provided that the number $g$ of distinct principal curvatures remains constant. Shortly thereafter, He, Xu, and Zhao~\cite{hxz26} removed the restriction on $g$, obtaining the following result.

\vspace{2mm}

\noindent
\textbf{Theorem} (He-Xu-Zhao). \emph{	Let $M^4$ be a closed minimal hypersurface immersed in $\mathbb S^5$. If
	$S$ and $P_3$ are constant, then $M$ is isoparametric.}
\vspace{2mm}

%\begin{thm}[He-Xu-Zhao]\label{thm:HXZ}
%	Let $M^4$ be a closed minimal hypersurface immersed in $\mathbb S^5$. If
%	$S$ and $P_3$ are constant, then $M$ is isoparametric.
%\end{thm}

The Chern conjecture predicts that, the two baseline conditions
``$
P_1=H=0, P_2=S=\mathrm{constant}
$"
should already force the entire principal-curvature structure to be rigid.
From this viewpoint, additional higher-order trace conditions should be
regarded as intermediate stepping stones toward the conjecture, rather than as intrinsic
features of the expected rigidity.
While the theorem above marks a major advance, it leaves open the question of whether the consecutive trace assumption $\{P_1, P_2, P_3\}$ is unique in yielding rigidity. To address this, we investigate the next higher-order symmetric polynomial, $P_4$. For a minimal $4$-dimensional hypersurface, Newton's identities yield the algebraic relation
$$
P_4=\frac12S^2-4K,
$$
where $K=\det A=\lambda_1\lambda_2\lambda_3\lambda_4$ is the Gauss-Kronecker curvature. Consequently, when $S$ is constant, the constancy of $P_4$ is precisely equivalent to the constancy of $K$. Geometrically, the Gauss-Kronecker curvature $K$ is a highly natural object: because the dimension is even, $K$ is invariant under a change of orientation of the normal bundle, making it globally well-defined even if the hypersurface $M$ is non-orientable.

In this paper, we resolve this nonconsecutive rigidity problem in the affirmative, establishing the following main result.

\begin{thm}\label{thm:main}
	Let $M^4$ be a closed minimal hypersurface immersed in $\mathbb S^5$. If the squared norm of the second fundamental form $S$ and the Gauss-Kronecker curvature $K$ are constant, then $M$ is isoparametric.  More precisely, 
    \begin{itemize}
        \item[(1)] If $K<0$, then $M$ is the Clifford torus $\mathbb{S}^{1}\left(\frac{1}{2}\right)\times\mathbb{S}^{3}\left(\frac{\sqrt{3}}{2}\right)$;
        \item[(2)] If $K=0$, then $M$ is the totally geodesic sphere $\mathbb{S}^{4}(1)$;
        \item[(3)] If $K>0$, let $\Omega=\{p\in M|g(p)=4\}\subset M^4$. 
        \begin{itemize}
            \item[(a)] If $\Omega=\emptyset$, then $M$ is the Clifford torus $\mathbb{S}^{2}\left(\frac{\sqrt{2}}{2}\right)\times\mathbb{S}^{2}\left(\frac{\sqrt{2}}{2}\right)$;
            \item[(b)] If $\Omega\neq \emptyset$, then $M$ is the isoparametric hypersurface with four distinct principal curvatures, which has $R_{M}\equiv 0$.  
        \end{itemize}
    \end{itemize}
\end{thm}

\begin{rem}
	The $K=0$ case was resolved by Cui~\cite{cui25}. For the completeness, we list it here.
\end{rem}

Theorem~\ref{thm:main} demonstrates that consecutive trace conditions are not the only possible source of rigidity. In dimension $4$, the nonconsecutive set $\{P_1, P_2, P_4\}$ is sufficient to completely freeze the principal-curvature structure, without any sign assumption on $R_M$. The algebraic transparency of this rigidity can be seen via the characteristic polynomial of the shape operator $A$:
$$
\det(xI-A)=x^4-\frac{S}{2}x^2-\frac{f}{3}x+K,
$$
where $f=P_3$. Under our hypotheses, both $S$ and $K$ are constant, meaning that all possible variation of the principal curvatures is encoded in the
single function $f$. 
%The analytic core of the proof is to show that $f$ is constant, which immediately implies that the principal curvatures are constant, and thus $M^4$ is isoparametric.

To achieve this, we must overcome two formidable challenges, which constitute the core technical innovations of this work:

\begin{itemize}
	\item \textbf{Construction of Two Novel Weighted $3$-Forms:} 
	The differential forms controlling $df$ are defined only on the open set $\Omega \subset M$ where the four principal curvatures are pairwise distinct, and they exhibit singular behavior near the multiple-root locus. We overcome this by constructing \emph{two new weighted $3$-forms} adapted to $S$
	and $K$. By carefully balancing these two forms, we extract crucial sign-definite properties that allow us to carry out a Tang-Yan type cut-off argument.
	
	\item \textbf{Topological Resolution of Analytical Obstacles:} 
	The analytic estimates required to close our weighted-form inequalities need a global curvature relation that cannot be derived from local differential identities. Our key conceptual innovation is to obtain the necessary global curvature relation \emph{topologically} rather than imposing it as an assumption. Specifically, we prove the vanishing of the Euler characteristic $\chi(M)$ by analyzing the Euler classes of two rank-$2$ vector bundles associated with the principal distributions. The Gauss-Bonnet-Chern formula then translates this topological vanishing into exactly the integral curvature identity needed to close the weighted-form estimates.
\end{itemize}

This interplay between the local geometry of the principal curvatures and
the global topology of their principal distributions suggests a possible direction
beyond the $4$-dimensional problem. In attempts to weaken or rearrange the
trace assumptions in higher-dimensional rigidity theorems, the global topology of the principal distributions may act as a powerful compensator, providing geometric constraints that are invisible to local differential analysis alone.

The paper is organized as follows. In Section \ref{sec 2}, we establish our notational and frame-field conventions, and recall the fundamental structure equations and algebraic identities. Section \ref{sec 3} is devoted to the proof of Theorem \ref{thm:main}. Specifically, we treat the case $K<0$ in Subsection \ref{3.1}. %, while the case $K=0$ is resolved by invoking a result of Cui~\cite{cui25}. 
For the most delicate case $K>0$, we construct two novel weighted $3$-forms and derive their key algebraic properties in Subsection \ref{3.2}. In Subsection~\ref{3.3}, we perform a localized cut-off estimates, which, combined with Assertion \ref{asser} and the vanishing of the Euler characteristic which is guaranteed by Proposition \ref{prop euler-vanishing}, completes the proof of Theorem \ref{thm:main}.  Subsections \ref{3.4} and \ref{3.5} are dedicated to the proof of Proposition \ref{prop euler-vanishing} and Assertion \ref{asser}, respectively.

\addvspace{2\baselineskip} % extra 2 blank lines before this section

\section{Preliminaries and Structural Equations}\label{sec 2}

Throughout this paper, all manifolds under consideration are assumed to be connected. Let $M^{n}$ be an immersed hypersurface in the unit sphere $\mathbb{S}^{n+1}$. For any $p\in M$, we choose a local orthonormal frame field $\{e_{1}, \dots, e_{n}\}$ in a neighborhood of $p$, with $\{\omega_{1}, \dots, \omega_{n}\}$ denoting its dual coframe field. The structural equations of $M$ are:
\begin{equation}\label{streq1}
	\begin{aligned}
		d\omega_{i} &= \sum_{j=1}^n \omega_{ij} \wedge \omega_{j},\quad \omega_{ij}+\omega_{ji}=0, \\
		d\omega_{ij} &= \sum_{k=1}^n \omega_{ik} \wedge \omega_{kj} - R_{ij},
	\end{aligned}
\end{equation}
where $\omega = (\omega_{ij})$ represents the connection $1$-forms and $R = (R_{ij})$ represents the curvature $2$-forms. The curvature tensor $R_{ijkl}$ is defined locally via
\begin{equation}\label{eq curvature form Rij}
	R_{ij} = \frac{1}{2}\sum_{k,l=1}^n R_{ijkl}\, \omega_k \wedge \omega_l.
\end{equation}
In matrix notation, the second structural equation in \eqref{streq1} is compactly expressed as 
\begin{equation*}\label{eq 2nd str eq}
	d\omega = \omega \wedge \omega - R.
\end{equation*}

The second fundamental form of $M$ is a symmetric bilinear form given by
$$
h = \sum_{i,j=1}^n h_{ij}\, \omega_i \otimes \omega_j,
$$
whose associated shape operator is represented by the symmetric matrix $A = (h_{ij})$. Its covariant derivative, denoted by $h_{ijk}$, is defined by the relation
\begin{equation}\label{eq second derivative}
	\sum_{m=1}^n h_{ijm}\, \omega_m = dh_{ij} + \sum_{m=1}^n h_{mj}\, \omega_{mi} + \sum_{m=1}^n h_{im}\, \omega_{mj}.
\end{equation}
With these conventions, the Gauss and Codazzi equations are formulated as
\begin{align}
	R_{ijkl} &= \delta_{ik}\delta_{jl} - \delta_{il}\delta_{jk} + h_{ik}h_{jl} - h_{il}h_{jk}, \label{gauss} \\
	h_{ijk} &= h_{ikj}. \label{codazzi}
\end{align}

Let $\lambda_1(p) \le \lambda_2(p) \le \dots \le \lambda_n(p)$ denote the principal curvatures of $M$ (the eigenvalues of $A$ at $p$). % and let $g$ be the number of distinct principal curvatures. 
The mean curvature $H$, the squared norm of the second fundamental form $S$, and the Gauss-Kronecker curvature $K$ are defined respectively by
$$
H = \tr A = \sum_{i=1}^n \lambda_{i}, \qquad S = \tr A^2 = \sum_{i=1}^n \lambda_i^2, \qquad K = \det A = \prod_{i=1}^n \lambda_{i}.
$$
Thus, $M$ is minimal if and only if $H \equiv 0$. 

For each $k \ge 1$, the power sum function $P_k := \tr A^k$ is smooth on $M$. For convenience, we denote the third power sum by $f := P_{3} = \tr A^3$. If $M$ is a minimal hypersurface, Newton's identities yield the following algebraic relations between the power sums $P_k$ and the $k$-th mean curvatures (the $k$-th elementary symmetric polynomial of principal curvatures) $H_k = \sum_{i_{1}<\dots<i_{k}}\lambda_{i_1}\dots\lambda_{i_k}$:
\begin{equation}\label{eq condition}
	\begin{aligned}
		P_{1} &= \sum_{i=1}^n \lambda_{i} =H_{1}=H=0,\\
		P_{2} &= \sum_{i=1}^n \lambda_{i}^2 =H_{1}^2-2H_{2}=-2H_{2}=S,\\
		P_{3} &= \sum_{i=1}^n \lambda_{i}^3= H_{1}^3-3H_{1}H_{2}+3H_{3}=3H_{3}=f,\\
		P_{4} &= \sum_{i=1}^n \lambda_{i}^4 =H_{1}^4-4H_{1}^2H_{2}+4H_{1}H_{3}+2H_{2}^2-4H_{4}=\frac{1}{2}S^2 -4H_{4}.
	\end{aligned}
\end{equation}
When $n=4$, we have $K=H_{4}$. In particular, if $M^4$ is minimal with constant $S$,  the constancy of $K$ is equivalent to the constancy of $P_{4}$.

\section{Global Rigidity and Topological Resolution}\label{sec 3}

This section is devoted to the proof of Theorem \ref{thm:main}. %To guide the reader, we briefly outline the analytical and topological arguments. 
We first establish the rigidity result for the case $K < 0$ in Subsection~\ref{3.1}. The remainder of this section addresses the more delicate case where $K > 0$. The analytical frameworks developed in Subsections~\ref{3.2} and~\ref{3.3} are initially formulated under the assumption that $M$ is orientable. For non-orientable manifolds, we lift the geometric data to the orientation double cover $\widetilde{M}$ in Subsection~\ref{3.4}. Consequently, Theorem \ref{thm:main} holds unconditionally regarding orientability. 
%\end{rem}

%========================================================================

\subsection{Rigidity under $K<0$.}\label{3.1}

%In this subsection, we establish the rigidity of closed minimal hypersurfaces in $\mathbb{S}^5$ under the negative curvature assumption $K < 0$. 
Let $M^4 \subset \mathbb{S}^5$ be a closed minimal hypersurface with constant $S$ and constant Gauss-Kronecker curvature $K < 0$. By utilizing the classic Peng-Terng identities and exploiting the constancy of these fundamental geometric quantities, we derive severe restrictions on the multiplicities of the principal curvatures.

We begin by recalling several pointwise identities on $M$. By the covariant derivative formula \eqref{eq second derivative}, for any local orthonormal frame and indices $i,j \in \{1,2,3,4\}$, we have
\begin{equation}\label{eq ejhii}
	e_j(h_{ii}) = h_{iij} - 2\sum_{m\neq i}h_{mi}\,\omega_{mi}(e_j).
\end{equation}
Since $M$ is minimal (i.e., $H \equiv 0$), differentiating the mean curvature along any direction $e_k$ yields
\begin{equation}\label{eq ekH=0}
	e_k(H) = \sum_{i=1}^4 e_k(h_{ii}) = \sum_{i=1}^4 \left( h_{iik} - 2\sum_{m\neq i}h_{mi}\,\omega_{mi}(e_k) \right) = 0.
\end{equation}
Furthermore, the constancy of $S = \sum_{i,j}h_{ij}^2$ implies
\begin{equation*}
	\frac{1}{2}e_k(S) = \sum_{i=1}^4 h_{ii}e_k(h_{ii}) + \sum_{i\neq j}h_{ij}e_k(h_{ij}) = 0,
\end{equation*}
		which, in view of \eqref{eq ejhii}, simplifies to the relation
		\begin{equation}\label{eq ekS}
			\sum_{i=1}^4 h_{ii}h_{iik} - 2\sum_{i=1}^4 h_{ii}\sum_{m\neq i}h_{mi}\,\omega_{mi}(e_k) + \sum_{i\neq j}h_{ij}e_k(h_{ij}) = 0.
		\end{equation}
		Analogously, since the fourth power sum $P_4 = \tr A^4 = \sum_{i,j,m,l}h_{ij}h_{jm}h_{ml}h_{li}$ is constant, its directional derivative vanishes:
		\begin{equation*}
			\begin{aligned}
				e_k(P_{4}) = &\sum_{i,j,m,l}e_k(h_{ij})h_{jm}h_{ml}h_{li} + \sum_{i,j,m,l}h_{ij}e_k(h_{jm})h_{ml}h_{li} \\
				&+ \sum_{i,j,m,l}h_{ij}h_{jm}e_k(h_{ml})h_{li} + \sum_{i,j,m,l}h_{ij}h_{jm}h_{ml}e_k(h_{li}) = 0.
			\end{aligned}
		\end{equation*}
		Exploiting the symmetry of the indices, all four summands in the expression above are mutually equal. Consequently, we obtain the simplified constraint
		\begin{equation}\label{eq:31-P4-constant}
			\sum_{i,j,m,l}e_k(h_{ij})h_{jm}h_{ml}h_{li} = 0.
		\end{equation}

		Now, fix an arbitrary point $p \in M$ and diagonalize the shape operator $A$ at $p$ such that $h_{ij}(p) = \lambda_i(p)\delta_{ij}$. At this point, equation \eqref{eq ejhii} reduces to $e_j(h_{ii}) = h_{iij}$, and the system of constraints \eqref{eq ekH=0}--\eqref{eq:31-P4-constant} simplifies for each index $k$ to the following linear system:
		\begin{equation}\label{eq:31-linear-system}
			\begin{cases}
				h_{11k} + h_{22k} + h_{33k} + h_{44k} = 0, \\
				\lambda_1 h_{11k} + \lambda_2 h_{22k} + \lambda_3 h_{33k} + \lambda_4 h_{44k} = 0, \\
				\lambda_1^3 h_{11k} + \lambda_2^3 h_{22k} + \lambda_3^3 h_{33k} + \lambda_4^3 h_{44k} = 0.
			\end{cases}
		\end{equation}
		
		With these algebraic relations in hand, we are ready to prove the main rigidity result for the negative curvature case.
		
		\begin{prop}\label{prop:K-negative}
			Let $M^4$ be a closed minimal hypersurface immersed in $\mathbb{S}^5$ with constant $S$ and constant $K < 0$. Then $S = 4$, and $M$ is congruent to the Clifford torus
			$$
			\mathbb{S}^1\left(\frac{1}{2}\right) \times \mathbb{S}^3\left(\frac{\sqrt{3}}{2}\right) \subset \mathbb{S}^5.
			$$
		\end{prop}

\begin{proof}
Let $g$ denote the number of distinct principal curvatures of $M^4$. We analyze the pointwise geometric behavior of $M$ based on the stratification of $g$. 			

If $g = 1$ at some point, the minimality of $M$ forces all principal curvatures to vanish identically, yielding $K = 0$. This contradicts the assumption $K < 0$, and thus we must have $g > 1$ globally on $M$.

Our proof relies on the classical Peng-Terng identities \cite{pt83}:
\begin{equation}\label{eq:31-PT}
	\Delta S = 2(4-S)S + 2\sum_{i,j,k}h_{ijk}^2, \qquad
	\Delta P_4 = 4(4-S)P_4 + 4(2\mathscr{A} + \mathscr{B}),
\end{equation}
where the terms $\mathscr{A}$ and $\mathscr{B}$ are defined by
$$
\mathscr{A} = \sum_{i,j,k}\lambda_i^2 h_{ijk}^2, \qquad \mathscr{B} = \sum_{i,j,k}\lambda_{i}\lambda_{j}h_{ijk}^2.
$$
Under the minimality condition $H = 0$, Newton's identities yield the relation
$$
P_4 = \frac{1}{2}S^2 - 4K.
$$
Since both $S$ and $K$ are constant, $S$ and $P_4$ are harmonic on $M$. Consequently, applying \eqref{eq:31-PT}, we obtain the key identity
\begin{equation}\label{eq:31-key-identity}
	4(4-S)K = -\frac{1}{2}S\sum_{i,j,k}h_{ijk}^2 + 2\mathscr{A} + \mathscr{B}.
\end{equation}
In what follows, we perform local calculations at an arbitrary point $p \in M$ using an orthonormal frame that diagonalizes the shape operator.

\medskip
\noindent\textbf{Case 1: $g(p) = 2$.} Since $K = \lambda_1\lambda_2\lambda_3\lambda_4 < 0$, we can, up to a change of orientation, arrange the principal curvatures as
$$
\lambda_1 = \lambda_2 = \lambda_3 = -\lambda, \qquad \lambda_4 = 3\lambda \qquad (\lambda > 0).
$$
Substituting these eigenvalues into the system \eqref{eq:31-linear-system} immediately yields $h_{44k} = 0$ for all $k$. Consequently, the components of the covariant derivative of the second fundamental form decompose as
\begin{align*}
	\sum_{i,j,k}h_{ijk}^2 &= \sum_{i,j,k\le3}h_{ijk}^2 + 3\sum_{i,j\le 3}h_{ij4}^2, \\
	\mathscr{A} &= \lambda^2 \sum_{i,j,k\le3} h_{ijk}^2 + 11\lambda^2 \sum_{i,j\le 3} h_{ij4}^2, \\
	\mathscr{B} &= \lambda^2 \sum_{i,j,k\le3} h_{ijk}^2 - 5\lambda^2 \sum_{i,j\le 3} h_{ij4}^2.
\end{align*}
Inserting these relations into \eqref{eq:31-key-identity} yields
$$
4(4-S)K = -\lambda^2 \left( 3\sum_{i,j,k\le3}h_{ijk}^2 + \sum_{i,j\le3}h_{ij4}^2 \right) \le 0.
$$
Since $K < 0$, this inequality implies $S \le 4$. By the classification theorems of Simons \cite{sim68}, Chern-do Carmo-Kobayashi \cite{cdk70}, and Lawson \cite{law69}, $M$ is either totally geodesic or a Clifford torus. The totally geodesic case is ruled out by the assumption $K < 0$. Thus, $M$ is congruent to the Clifford torus  $\mathbb{S}^1(\frac{1}{2}) \times \mathbb{S}^3(\frac{\sqrt{3}}{2})$, which satisfies $S = 4$.

			\medskip
			\noindent\textbf{Case 2: $g(p) = 3$.} Since $K < 0$, up to orientation, the principal curvatures can be arranged as $\lambda_{1}=\lambda_{2}<\lambda_{3}<0<\lambda_{4}$ or $\lambda_{1}<\lambda_{2}=\lambda_{3}<0<\lambda_{4}$. Actually, the calculation is independent of the order of the principal curvatures, so we treat the former case. Assume 
			$$	\lambda_1=\lambda_2=-\lambda,\qquad
			\lambda_3=-\mu,\qquad
			\lambda_4=2\lambda+\mu,
			\qquad (\lambda,\mu>0,\,\, \lambda\ne\mu).$$
			In this configuration, the system \eqref{eq:31-linear-system} reduces to
			\begin{equation}\label{eq:31-g3-derivatives}
				h_{11k} + h_{22k} = 0, \qquad h_{33k} = h_{44k} = 0.
			\end{equation}
			for each $k$. Applying \eqref{eq:31-g3-derivatives}, a direct calculation shows that
			\begin{equation*}
				\begin{aligned}
					\sum_{i,j,k}h_{ijk}^2 &= 4(h_{111}^2 + h_{112}^2) + 6(h_{113}^2 + h_{114}^2 + h_{123}^2 + h_{124}^2 + h_{134}^2 + h_{234}^2), \\
					\mathscr{A} &= 4\lambda^2(h_{111}^2 + h_{112}^2) + (4\lambda^2 + 2\mu^2)(h_{113}^2 + h_{123}^2) \\
					&\quad + (12\lambda^2 + 8\lambda\mu + 2\mu^2)(h_{114}^2 + h_{124}^2) + (10\lambda^2 + 8\lambda\mu + 4\mu^2)(h_{134}^2 + h_{234}^2), \\
					\mathscr{B} &= 4\lambda^2(h_{111}^2 + h_{112}^2) + (2\lambda^2 + 4\lambda\mu)(h_{113}^2 + h_{123}^2) - (6\lambda^2 + 4\lambda\mu)(h_{114}^2 + h_{124}^2) \\
					&\quad - (4\lambda^2 + 4\lambda\mu + 2\mu^2)(h_{134}^2 + h_{234}^2).     
				\end{aligned}
			\end{equation*}
		Substituting these expressions into the key identity \eqref{eq:31-key-identity}, we find
			\begin{equation*}
				\begin{aligned}
					4(4-S)K = &-(4\mu^2 + 8\lambda\mu)(h_{111}^2 + h_{112}^2) - (8\lambda^2 + 8\lambda\mu + 2\mu^2)(h_{113}^2 + h_{123}^2) \\
					&- 2\mu^2(h_{114}^2 + h_{124}^2) - 2\lambda^2(h_{134}^2 + h_{234}^2) \le 0.
				\end{aligned}
			\end{equation*}
As before, since $K < 0$, we obtain $S \le 4$, which implies that $M$ is either totally geodesic or a Clifford torus. However, neither of these homogeneous models admits a point with three distinct principal curvatures, which is a contradiction. Consequently, the case $g(p) = 3$ cannot occur.

\medskip
\noindent\textbf{Case 3: $g \equiv 4$ on $M$.} Since $M$ is compact, the smooth function $f = P_3 = \tr A^3$ achieves its maximum at some point $p \in M$. At this critical point, the condition $df = 0$ yields
$$
\lambda_1^2 h_{11k} + \lambda_2^2 h_{22k} + \lambda_3^2 h_{33k} + \lambda_4^2 h_{44k} = 0
$$
for each $k$. Together with \eqref{eq:31-linear-system}, this forms a Vandermonde-type system in terms of $\{h_{11k}, h_{22k}, h_{33k}, h_{44k}\}$. Since the four principal curvatures are pairwise distinct at $p$, the system has only the trivial solution:
$$
h_{11k} = h_{22k} = h_{33k} = h_{44k} = 0 \qquad \text{at } p.
$$

Consequently, only components $h_{ijk}$ with pairwise distinct indices can contribute to the identity \eqref{eq:31-key-identity} at $p$. Under this constraint, a direct algebraic computation reveals:
\begin{equation*}
		\begin{aligned}
					4(4-S)K &= \sum_{\substack{i,j,k \\ \text{pairwise distinct}}} \left( -\frac{1}{2}S + 2\lambda_i^2 + \lambda_i\lambda_j \right) h_{ijk}^2 \\
					&= \sum_{i<j<k} \left( -3S + 4\lambda_i^2 + 4\lambda_j^2 + 4\lambda_k^2 + 2\lambda_i\lambda_j + 2\lambda_i\lambda_k + 2\lambda_j\lambda_k \right) h_{ijk}^2 \\
					&= \sum_{i<j<k} \left( -3S + 4\lambda_i^2 + 4\lambda_j^2 + 4\lambda_k^2 + 2H_2 + 2\lambda_{10-i-j-k}^2 \right) h_{ijk}^2 \\
					&= -2\left( \lambda_1^2 h_{234}^2 + \lambda_2^2 h_{134}^2 + \lambda_3^2 h_{124}^2 + \lambda_4^2 h_{123}^2 \right) \le 0.
		\end{aligned}
\end{equation*}
Once again, the assumption $K < 0$ forces $S \le 4$, leading to the same contradiction since neither the sphere nor the Clifford tori allow $g \equiv 4$ globally.

\medskip
In conclusion, the only geometrically realizable scenario is $g \equiv 2$ with $S = 4$. Consequently, $M$ is uniquely determined as the Clifford torus $\mathbb{S}^1(1/2) \times \mathbb{S}^3(\sqrt{3}/2)$.

		\end{proof}

%-------------------------------------------------------------------------------

\subsection{Construction and Exterior Differentiation of Weighted $(n-1)$-forms}\label{3.2}

In their pioneering work, de Almeida and Brito \cite{ab90} introduced a delicate $2$-form to study the $3$-dimensional Chern conjecture, a construction later generalized by Tang and Yan \cite{ty23} to an $(n-1)$-form in arbitrary dimension $n$. Inspired by these developments, we incorporate carefully chosen weights into these differential forms, establishing two novel weighted $3$-forms designed specifically to resolve the positive curvature case $K > 0$.

Let $M^n$ be an oriented hypersurface immersed in $\mathbb{S}^{n+1}$. We denote by
$$
\Omega = \big\{ p \in M \mid \lambda_1(p) < \lambda_2(p) < \dots < \lambda_n(p) \big\}
$$
the open subset of $M$ where the principal curvatures are pairwise distinct. On $\Omega$, we choose a local oriented orthonormal coframe field $\{\omega_i\}_{i=1}^n$ dual to the principal directions, so that the volume form and the second fundamental form are given by
$$
\operatorname{vol} = \omega_1 \wedge \dots \wedge \omega_n, \qquad h = \sum_{i=1}^n \lambda_i\, \omega_i \otimes \omega_i.
$$
On this open set, the principal curvatures $\lambda_i$ vary smoothly on $\Omega$ and satisfy the classical differential relations:
\begin{equation}\label{eq:eigenvalue-derivative}
	e_i(\lambda_j) = h_{jji},
\end{equation}
and, for $i \neq j$, the connection $1$-forms are given by
\begin{equation}\label{eq:connection-simple-spectrum}
	\omega_{ij} = \sum_{m=1}^n \frac{h_{ijm}}{\lambda_i - \lambda_j}\, \omega_m.
\end{equation}

Now, let $U \subset \mathbb{R}$ be an open interval containing the image of all principal curvature functions $\lambda_i$ on the connected component of $\Omega$ under consideration. For any smooth function $\alpha \in C^\infty(U)$, we define the symmetric coefficients on $\Omega$ for $i \neq j$ by
$$
\alpha_{ij} := \frac{\alpha(\lambda_i) - \alpha(\lambda_j)}{\lambda_i - \lambda_j}.
$$
Using these coefficients, we define the weighted $(n-1)$-form $\Psi$ on $\Omega$ by
\begin{equation}\label{eq:Psi-definition}
	\Psi = \sum_{1 \le i < j \le n} (-1)^{i+j-1} \alpha_{ij}\, \omega_1 \wedge \dots \wedge \widehat{\omega}_i \wedge \dots \wedge \widehat{\omega}_j \wedge \dots \wedge \omega_n \wedge \omega_{ij}.
\end{equation}
This differential form is globally well-defined on $\Omega$, independent of the choice of the adapted principal coframe. Indeed, any transition between oriented adapted coframes must be of the form $\omega_i \mapsto \varepsilon_i \omega_i$ with $\varepsilon_i = \pm 1$ and $\prod_{i=1}^n \varepsilon_i = 1$. Under this gauge transformation, the connection forms transform as $\omega_{ij} \mapsto \varepsilon_i \varepsilon_j \omega_{ij}$ for $i \neq j$, leaving each summand in \eqref{eq:Psi-definition} invariant.

This general construction unifies and generalizes several classical differential forms in the literature. Specifically, setting the linear weight $\alpha(x) = x$ recovers the $(n-1)$-form introduced by Tang and Yan \cite{ty23}, whereas the quadratic weight $\alpha(x) = x^2$ yields the $(n-1)$-form subsequently exploited in \cite{cl25} and \cite{hxz26}.

\begin{rem}
	The local smoothness of $\alpha$ on $U$ is crucial. In particular, our subsequent choice $\alpha(x) = -Kx^{-1}$ is mathematically rigorous only on the region where all  the principal curvatures are non-vanishing. This condition is naturally guaranteed under the assumption $K > 0$.
\end{rem}

We now compute the exterior derivative $d\Psi$. For simplicity, we denote the basis $(n-2)$-forms by
$$
\eta_{ij} := \omega_1 \wedge \dots \wedge \widehat{\omega}_i \wedge \dots \wedge \widehat{\omega}_j \wedge \dots \wedge \omega_n \qquad (i < j).
$$
This allows us to decompose the exterior derivative as $d\Psi = I + II + III$, where
\begin{align*}
	I &= \sum_{i<j} (-1)^{i+j-1} \alpha_{ij}\, d\eta_{ij} \wedge \omega_{ij}, \\
	II &= \sum_{i<j} (-1)^{i+j-1+n-2} \alpha_{ij}\, \eta_{ij} \wedge d\omega_{ij}, \\
	III &= \sum_{i<j} (-1)^{i+j-1} d\alpha_{ij} \wedge \eta_{ij} \wedge \omega_{ij}.
\end{align*}
Utilizing the structural equations \eqref{streq1}, the Gauss-Codazzi relations \eqref{gauss}--\eqref{codazzi}, and the connection formula \eqref{eq:connection-simple-spectrum}, the first term $I$ expands to
$$
I = \sum_{i<j} (-1)^n \alpha_{ij} \sum_{k \neq i,j} \left[ \frac{h_{kki}h_{jji}}{(\lambda_k - \lambda_i)(\lambda_j - \lambda_i)} + \frac{h_{kkj}h_{iij}}{(\lambda_k - \lambda_j)(\lambda_i - \lambda_j)} - \frac{h_{ijk}^2}{(\lambda_i - \lambda_k)(\lambda_k - \lambda_j)} \right] \operatorname{vol}.
$$
Here, the final term in the bracket is obtained by combining the two components
$$
\frac{h_{ijk}^2}{(\lambda_k - \lambda_i)(\lambda_i - \lambda_j)} - \frac{h_{ijk}^2}{(\lambda_k - \lambda_j)(\lambda_i - \lambda_j)} = -\frac{h_{ijk}^2}{(\lambda_i - \lambda_k)(\lambda_k - \lambda_j)}.
$$
Crucially, the contribution of the terms involving $h_{ijk}^2$ vanishes upon summation over all indices. Indeed, for any triple of distinct indices $i < j < k$, the definition of $\alpha_{pq}$ yields the algebraic identity
$$
\frac{\alpha_{ij}}{(\lambda_i - \lambda_k)(\lambda_k - \lambda_j)} + \frac{\alpha_{ik}}{(\lambda_i - \lambda_j)(\lambda_j - \lambda_k)} + \frac{\alpha_{jk}}{(\lambda_j - \lambda_i)(\lambda_i - \lambda_k)} = 0.
$$
Thus, the term $I$ simplifies to
$$
I = \sum_{i<j} (-1)^n \alpha_{ij} \sum_{k \neq i,j} \left[ \frac{h_{kki}h_{jji}}{(\lambda_k - \lambda_i)(\lambda_j - \lambda_i)} + \frac{h_{kkj}h_{iij}}{(\lambda_k - \lambda_j)(\lambda_i - \lambda_j)} \right] \operatorname{vol}.
$$

Similarly, using the Gauss curvature relation $R_{ijij} = 1 + \lambda_i \lambda_j$ for hypersurfaces in the unit sphere, the second term $II$ is written as
\begin{equation*}
	II = \sum_{i<j} (-1)^n \alpha_{ij} \left[ \sum_{k \neq i,j} \left( \frac{h_{iik}h_{jjk}}{(\lambda_i - \lambda_k)(\lambda_k - \lambda_j)} - \frac{h_{ijk}^2}{(\lambda_i - \lambda_k)(\lambda_k - \lambda_j)} \right) - R_{ijij} \right] \operatorname{vol}.
\end{equation*}
The $h_{ijk}^2$ terms in $II$ vanish by the same cyclic algebraic identity, yielding
\begin{equation*}
	II = \sum_{i<j} (-1)^n \alpha_{ij} \left[ \sum_{k \neq i,j} \frac{h_{iik}h_{jjk}}{(\lambda_i - \lambda_k)(\lambda_k - \lambda_j)} - R_{ijij} \right] \operatorname{vol}.
\end{equation*}
Finally, the contribution from the derivatives of the coefficients is given by
\begin{equation*}
	III = \sum_{i<j} (-1)^n \frac{(\alpha_{ij})_i h_{jji} - (\alpha_{ij})_j h_{iij}}{\lambda_i - \lambda_j} \operatorname{vol}.
\end{equation*}
Summing these three contributions, we arrive at the general formula for the exterior derivative of the weighted $(n-1)$-form:
\begin{equation}\label{eq:dPsi-general}
	\begin{aligned}
		d\Psi = \sum_{1 \le i < j \le n} (-1)^n \Bigg\{ &\alpha_{ij} \sum_{k \neq i,j} \left[ \frac{h_{kki} h_{jji}}{(\lambda_k - \lambda_i)(\lambda_j - \lambda_i)} + \frac{h_{kkj} h_{iij}}{(\lambda_k - \lambda_j)(\lambda_i - \lambda_j)} + \frac{h_{iik} h_{jjk}}{(\lambda_i - \lambda_k)(\lambda_k - \lambda_j)} \right] \\
		&+ \frac{1}{\lambda_i - \lambda_j} \Big( (\alpha_{ij})_i h_{jji} - (\alpha_{ij})_j h_{iij} \Big) - \alpha_{ij} R_{ijij} \Bigg\} \operatorname{vol}.
	\end{aligned}
\end{equation}

\medskip
We now specialize this construction to the case $n=4$. Let $M^4$ be an oriented, closed minimal hypersurface in $\mathbb{S}^5$ with constant $S$ and constant positive Gauss-Kronecker curvature $K > 0$. The condition $K = \lambda_1\lambda_2\lambda_3\lambda_4 > 0$ guarantees that the principal curvatures are nowhere zero on $M$. Consequently, the two choices of auxiliary functions
$$
\alpha(x) = -Kx^{-1} \quad \text{and} \quad \alpha(x) = x^3
$$
are smooth and globally well-defined on $\Omega$. These choices yield two key weighted $3$-forms, whose indices correspond directly to the algebraic degree of the generating functions $\alpha$:
\begin{equation}\label{eq psi1}
	\Psi_{-1} = \sum_{i<j} (-1)^{i+j-1} \frac{K}{\lambda_i \lambda_j}\, \theta_{ij},
\end{equation}
and
\begin{equation}\label{eq psi2}
	\Psi_3 = \sum_{i<j} (-1)^{i+j-1} (\lambda_i^2 + \lambda_i \lambda_j + \lambda_j^2)\, \theta_{ij},
\end{equation}
where the basis $3$-forms $\{\theta_{ij}\}$ are defined by
\begin{align*}
	\theta_{12} &= \omega_3 \wedge \omega_4 \wedge \omega_{12}, \quad & \theta_{13} &= \omega_2 \wedge \omega_4 \wedge \omega_{13}, \quad & \theta_{14} &= \omega_2 \wedge \omega_3 \wedge \omega_{14}, \\
	\theta_{23} &= \omega_1 \wedge \omega_4 \wedge \omega_{23}, \quad & \theta_{24} &= \omega_1 \wedge \omega_3 \wedge \omega_{24}, \quad & \theta_{34} &= \omega_1 \wedge \omega_2 \wedge \omega_{34}.
\end{align*}

For each $k \in \{1,2,3,4\}$, the constancy of $H$, $S$, and $P_4 = \frac{1}{2}S^2 - 4K$ yields:
$$
\begin{cases}
	h_{11k} + h_{22k} + h_{33k} = -h_{44k}, \\
	\lambda_1 h_{11k} + \lambda_2 h_{22k} + \lambda_3 h_{33k} = -\lambda_4 h_{44k}, \\
	\lambda_1^3 h_{11k} + \lambda_2^3 h_{22k} + \lambda_3^3 h_{33k} = -\lambda_4^3 h_{44k}.
\end{cases}
$$
Since the principal curvatures are pairwise distinct on $\Omega$, the coefficient matrix of $\{h_{11k}, h_{22k}, h_{33k}\}$ is non-singular. By virtue of Cramer's rule, we can express these components in terms of $h_{44k}$ as follows:
\begin{equation}\label{eq Cramer}
	\begin{aligned}
		h_{11k} &= \frac{\lambda_1(\lambda_3 - \lambda_4)(\lambda_2 - \lambda_4)}{\lambda_4(\lambda_1 - \lambda_3)(\lambda_2 - \lambda_1)}\, h_{44k}, \\
		h_{22k} &= \frac{\lambda_2(\lambda_1 - \lambda_4)(\lambda_3 - \lambda_4)}{\lambda_4(\lambda_1 - \lambda_2)(\lambda_2 - \lambda_3)}\, h_{44k}, \\
		h_{33k} &= \frac{\lambda_3(\lambda_1 - \lambda_4)(\lambda_4 - \lambda_2)}{\lambda_4(\lambda_1 - \lambda_3)(\lambda_2 - \lambda_3)}\, h_{44k}.
	\end{aligned}
\end{equation}

By substituting \eqref{eq Cramer} into the general exterior derivative formula \eqref{eq:dPsi-general} and incorporating the Gauss equations $R_{ijij} = 1 + \lambda_i \lambda_j$, we obtain the following expression for the derivative of $\Psi_{-1}$:
\begin{equation}\label{eq dPsi1}
	d\Psi_{-1} = \left\{ \sum_{i=1}^4 c_i h_{44i}^2 - 6K + \frac{1}{2}S \right\} \operatorname{vol} \qquad \text{on } \Omega,
\end{equation}
where the coefficients $c_i$ are defined by
$$
c_i = -4\gamma K \left( \lambda_j^4 + \lambda_k^4 + \lambda_l^4 - \lambda_j^2\lambda_k^2 - \lambda_j^2\lambda_l^2 - \lambda_k^2\lambda_l^2 \right)
$$
for any permutation $\{i,j,k,l\} = \{1,2,3,4\}$, with the positive factor $\gamma$ on $\Omega$ given by
$$
\gamma = \frac{1}{\lambda_4^2 (\lambda_1 - \lambda_2)^2 (\lambda_1 - \lambda_3)^2 (\lambda_2 - \lambda_3)^2} > 0.
$$
The constant term in \eqref{eq dPsi1} arises from the summation
$$
-\sum_{i<j} \frac{K}{\lambda_i\lambda_j}(1 + \lambda_i\lambda_j) = -6K + \frac{1}{2}S.
$$

To ensure self-containment and verify the sign of the coefficients, we explicitly compute the representative coefficient $c_1$:
$$
c_1 = \gamma K c'_1,
$$
where
\begin{align*}
	c'_1 =& -(\lambda_2 - \lambda_4)(\lambda_1 - \lambda_4)^2(\lambda_3 - \lambda_4) - (\lambda_2 - \lambda_4)(\lambda_1 - \lambda_4)^2(\lambda_2 - \lambda_3) \\
	&+ (\lambda_2 - \lambda_3)(\lambda_1 - \lambda_3)^2(\lambda_3 - \lambda_4) \\
	=& -4 \left( \lambda_2^4 + \lambda_3^4 + \lambda_4^4 - \lambda_2^2\lambda_3^2 - \lambda_2^2\lambda_4^2 - \lambda_3^2\lambda_4^2 \right).
\end{align*}
The second equality is obtained by substituting the minimality relation $\lambda_1 + \lambda_2 + \lambda_3 + \lambda_4 = 0$ into the algebraic expansion. For any three distinct indices $\{j,k,l\}$, we have
$$
\lambda_j^4 + \lambda_k^4 + \lambda_l^4 - \lambda_j^2\lambda_k^2 - \lambda_j^2\lambda_l^2 - \lambda_k^2\lambda_l^2 = \frac{1}{2} \left[ (\lambda_j^2 - \lambda_k^2)^2 + (\lambda_j^2 - \lambda_l^2)^2 + (\lambda_k^2 - \lambda_l^2)^2 \right] > 0
$$
on $\Omega$. Since $\gamma > 0$ and $K > 0$, it follows that $c_1 < 0$. By cyclic permutation of indices, we immediately obtain
\begin{equation}\label{eq:ci-negative}
	c_i < 0 \qquad \text{for all } i \in \{1,2,3,4\}.
\end{equation}

In a similar manner, the exterior derivative of the second weighted $3$-form $\Psi_3$ is computed as
\begin{equation}\label{eq dPsi2}
	d\Psi_3 = \left\{ \sum_{i=1}^4 d_i h_{44i}^2 + \frac{1}{4}S^2 - \frac{5}{2}S - 6K \right\} \operatorname{vol} \qquad \text{on } \Omega,
\end{equation}
where the coefficients $d_i$ are given by
$$
d_i = -\gamma \left[ \frac{S}{2\lambda_i^2} \left( \lambda_i^4 - \frac{S^2 + 16K}{2S}\lambda_i^2 + 3K \right)^2 + \frac{2K}{S}(S^2 - 16K)\lambda_i^2 \right].
$$
The constant term in \eqref{eq dPsi2} is obtained by summing
$$
-\sum_{i<j} (\lambda_i^2 + \lambda_i\lambda_j + \lambda_j^2)(1 + \lambda_i\lambda_j) = \frac{1}{4}S^2 - \frac{5}{2}S - 6K.
$$
We outline the reduction of the representative coefficient $d_1 = -\gamma d'_1$, where
\begin{align*}
	d'_1 &= 3\lambda_2^2\lambda_3^2\lambda_4^2 \left( \lambda_2^2 + \lambda_3^2 + \lambda_4^2 - \lambda_2\lambda_3 - \lambda_2\lambda_4 - \lambda_3\lambda_4 \right) \\
	&\quad - 6\lambda_1\lambda_2\lambda_3\lambda_4 \left( \lambda_2^2\lambda_3^2 + \lambda_2^2\lambda_4^2 + \lambda_3^2\lambda_4^2 - \lambda_2^2\lambda_3\lambda_4 - \lambda_2\lambda_3^2\lambda_4 - \lambda_2\lambda_3\lambda_4^2 \right) \\
	&\quad + \lambda_1^2 \Big( \lambda_2^4\lambda_3^2 + \lambda_2^2\lambda_3^4 + \lambda_2^4\lambda_4^2 + \lambda_2^2\lambda_4^4 + \lambda_3^4\lambda_4^2 + \lambda_3^2\lambda_4^4 + \lambda_2^3\lambda_3^3 + \lambda_2^3\lambda_4^3 + \lambda_3^3\lambda_4^3 \\
	&\qquad\quad - 2\lambda_2^4\lambda_3\lambda_4 - 2\lambda_2\lambda_3^4\lambda_4 - 2\lambda_2\lambda_3\lambda_4^4 + 3\lambda_2^2\lambda_3^2\lambda_4^2 \\
	&\qquad\quad - \lambda_2\lambda_3^2\lambda_4^3 - \lambda_2\lambda_3^3\lambda_4^2 - \lambda_2^2\lambda_3\lambda_4^3 - \lambda_2^3\lambda_3\lambda_4^2 - \lambda_2^2\lambda_3^3\lambda_4 - \lambda_2^3\lambda_3^2\lambda_4 \Big).
\end{align*}
To simplify $d'_1$, we introduce the elementary symmetric polynomials of the variables $\{\lambda_2, \lambda_3, \lambda_4\}$:
$$
\rho_1 = \lambda_2 + \lambda_3 + \lambda_4 = -\lambda_1, \quad \rho_2 = \lambda_2\lambda_3 + \lambda_2\lambda_4 + \lambda_3\lambda_4 = \lambda_1^2 - \frac{1}{2}S, \quad \rho_3 = \lambda_2\lambda_3\lambda_4 = \frac{K}{\lambda_1}.
$$
Under these algebraic substitutions, the expression for $d'_1$ reduces to
\begin{align*}
	d'_1 &= 3\rho_3^2(\rho_1^2 - 3\rho_2) - 6\lambda_1\rho_3(\rho_2^2 - 3\rho_1\rho_3) + \lambda_1^2(\rho_1^2\rho_2^2 - 4\rho_1^3\rho_3 + 6\rho_1\rho_2\rho_3 - \rho_2^3) \\
	&= \frac{1}{2}S\lambda_1^6 - \left( 8K + \frac{1}{2}S^2 \right) \lambda_1^4 + \left( 9KS + \frac{1}{8}S^3 \right) \lambda_1^2 - \left( 24K^2 + \frac{1}{2}KS^2 \right) + \frac{9K^2S}{2\lambda_1^2} \\
	&= \frac{S}{2\lambda_1^2} \left( \lambda_1^4 - \frac{S^2 + 16K}{2S}\lambda_1^2 + 3K \right)^2 + \frac{2K}{S}(S^2 - 16K)\lambda_1^2.
\end{align*}
On the hypersurface $M$, we have the general algebraic inequality
$$
S^2 = \left( \sum_{i=1}^4 \lambda_i^2 \right)^2 \ge 16|\lambda_1\lambda_2\lambda_3\lambda_4| = 16K,
$$
where the equality holds if and only if all $|\lambda_i|$ are identical. However, this homogeneous configuration is strictly excluded on the region $\Omega$. Consequently, the strict inequality $S^2 > 16K$ holds on $\Omega$, which guarantees $d_1 = -\gamma d'_1 < 0$. By symmetric permutation of the indices, we conclude that
\begin{equation}\label{eq:di-negative}
	d_i < 0 \qquad \text{for all } i \in \{1,2,3,4\} \text{ on } \Omega.
\end{equation}

\hfill $\Box$
%-----------------------------------------------------------------------------------------------------

\subsection{Rigidity under $K>0$. %: Spectral Stratification and Variational Cut-off
}\label{3.3}

In this subsection, we establish the rigidity of a closed minimal hypersurface $M^4 \subset \mathbb{S}^5$ in the positive curvature regime $K > 0$. The proof proceeds via a systematic combination of algebraic stratification, localized variational cut-off estimates, and topological obstruction theory.

We first analyze the characteristic polynomial of the shape operator, which yields a precise topological and analytical characterization of the open subset $\Omega \subset M$ where the four principal curvatures remain mutually distinct. Second, by employing the weighted $3$-forms $\Psi_{-1}$ and $\Psi_3$ constructed in Subsection~\ref{3.2}, we perform a localized cut-off analysis on collar neighborhoods of the singular strata $M \setminus \Omega$ to establish the global constancy of the third power sum $f = P_3=\tr A^3$ under either one of two auxiliary numerical inequalities. Finally, utilizing the Gauss-Bonnet-Chern formula coupled with the vanishing of the Euler characteristic $\chi(M) = 0$ (to be established in Subsection~\ref{3.4}), we demonstrate that these auxiliary inequalities unconditionally cover the entire positive regime $K > 0$.

\begin{lem}\label{lem lem3.3 figure}
	Let $M^4$ be a closed minimal hypersurface immersed in $\mathbb{S}^5$ with constant $S$ and constant $K > 0$. Let $\Omega = \{ p \in M \mid g(p) = 4 \}$ denote the open set where the four principal curvatures are pairwise distinct. Then there exists a non-negative constant $a \ge 0$, uniquely determined by $S$ and $K$, such that $f(M) \subset [-a, a]$. Furthermore, the following alternatives hold:
	\begin{enumerate}
		\item $\Omega = \emptyset$ if and only if $a = 0$. In this case, $f \equiv 0$ on $M$, and $M$ is congruent to the minimal Clifford torus
		$$
		\mathbb{S}^2\left(\frac{\sqrt{2}}{2}\right) \times \mathbb{S}^2\left(\frac{\sqrt{2}}{2}\right) \subset \mathbb{S}^5.
		$$
		\item $\Omega \neq \emptyset$ if and only if $a > 0$. In this case, the level sets of $f$ partition $M$ via
		$$
		f^{-1}(-a) \cup f^{-1}(a) = \{ p \in M \mid g(p) = 3 \}, \qquad f^{-1}(-a, a) = \Omega.
		$$
	\end{enumerate}
\end{lem}

\begin{proof}
	By virtue of the minimality condition $H = 0$ and Newton's identities \eqref{eq condition}, the characteristic polynomial of the shape operator $A$ is given explicitly by
	$$
	F(x) = \prod_{i=1}^4(x - \lambda_i) = x^4 - \frac{S}{2}x^2 - \frac{f}{3}x + K.
	$$
	Since $K > 0$, no principal curvature vanishes on $M$. Thus, for any value $\zeta \in f(M)$, the principal curvatures correspond to the four real roots of the algebraic equation
	\begin{equation}\label{eq fig}
		G(x) := x^3 - \frac{S}{2}x + \frac{K}{x} = \frac{\zeta}{3}.
	\end{equation}
	The rational function $G(x)$ is odd on $\mathbb{R} \setminus \{0\}$, with its first derivative given by
	$$
	G'(x) = 3x^2 - \frac{S}{2} - \frac{K}{x^2}.
	$$
	The critical points of $G$ are determined by the equation $6x^4 - Sx^2 - 2K = 0$, which admits exactly two real roots:
	\begin{equation}\label{eq x1,x2 extreme point}
		x_1 = \sqrt{\frac{S + \sqrt{S^2 + 48K}}{12}}, \qquad x_2 = -x_1.
	\end{equation}
	Since the second derivative evaluates to $G''(x) = 6x + 2Kx^{-3}$, we have $G''(x_1) > 0$ and $G''(x_2) < 0$, identifying $x_1$ as a non-degenerate strict local minimum and $x_2$ as a non-degenerate strict local maximum. The graph of $G'(x)$ is shown in Figure \ref{fig:G-prime}. Consequently, $G$ increases strictly on $(-\infty, x_2)$ and $(x_1, +\infty)$, and decreases strictly on $(x_2, 0)$ and $(0, x_1)$. Since $G''(x) \neq 0$ on $\mathbb{R} \setminus \{0\}$, the critical points are non-degenerate, implying that any root of \eqref{eq fig} has multiplicity at most $2$, with $x_1$ and $x_2$ being the only possible multiple roots.

	\begin{figure}[!htbp]
	\centering
	\includegraphics[scale=0.6]{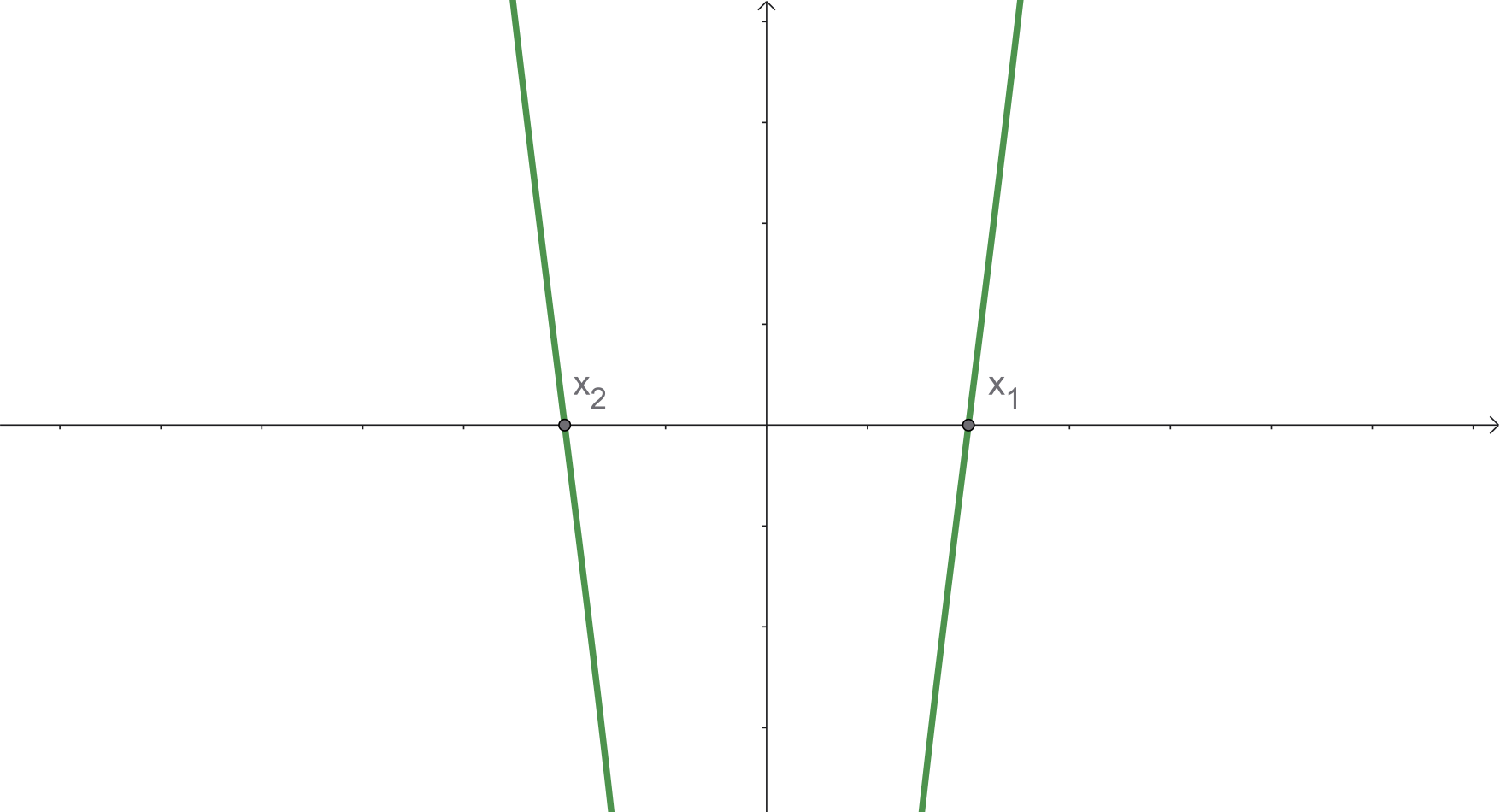}
	\caption{The graph of $G'$.}
	\label{fig:G-prime}
\end{figure}

For \eqref{eq fig} to possess four real roots (counted with multiplicity), the horizontal line $y = \zeta/3$ must intersect both branches of the graph of $G(x)$. By the monotonicity of $G$, this intersection occurs if and only if
$$
G(x_1) \le \frac{\zeta}{3} \le G(x_2),
$$
which implies $f(M) \subset [3G(x_1), 3G(x_2)]$. Since $f(M)$ is non-empty and $G$ is odd, we have $G(x_2) \ge 0 \ge G(x_1)$. We define
$$
a := 3G(x_2) = -3G(x_1) \ge 0.
$$

If $a = 0$, then $f(M) \subset \{0\}$, forcing $f \equiv 0$ on $M$, and thus $M$ is isoparametric. In this degenerate case, equation \eqref{eq fig} admits a double negative root $x_2$ and a double positive root $x_1 = -x_2$. Hence, the principal curvatures occur in opposite pairs $(\lambda, \lambda, -\lambda, -\lambda)$ with $\lambda > 0$. According to the classification of minimal isoparametric hypersurfaces in spheres, this $4$-dimensional minimal hypersurface with $g = 2$ and equal multiplicities must be congruent to the minimal Clifford torus $\mathbb{S}^2(\frac{\sqrt{2}}{2}) \times \mathbb{S}^2(\frac{\sqrt{2}}{2})$. This confirms that $a = 0$ is equivalent to $\Omega = \emptyset$. The corresponding graph of $G(x)$ is presented in Figure \ref{fig:G-degenerate} as follows

	\begin{figure}[!htbp]
	\centering
	\includegraphics[scale=0.8]{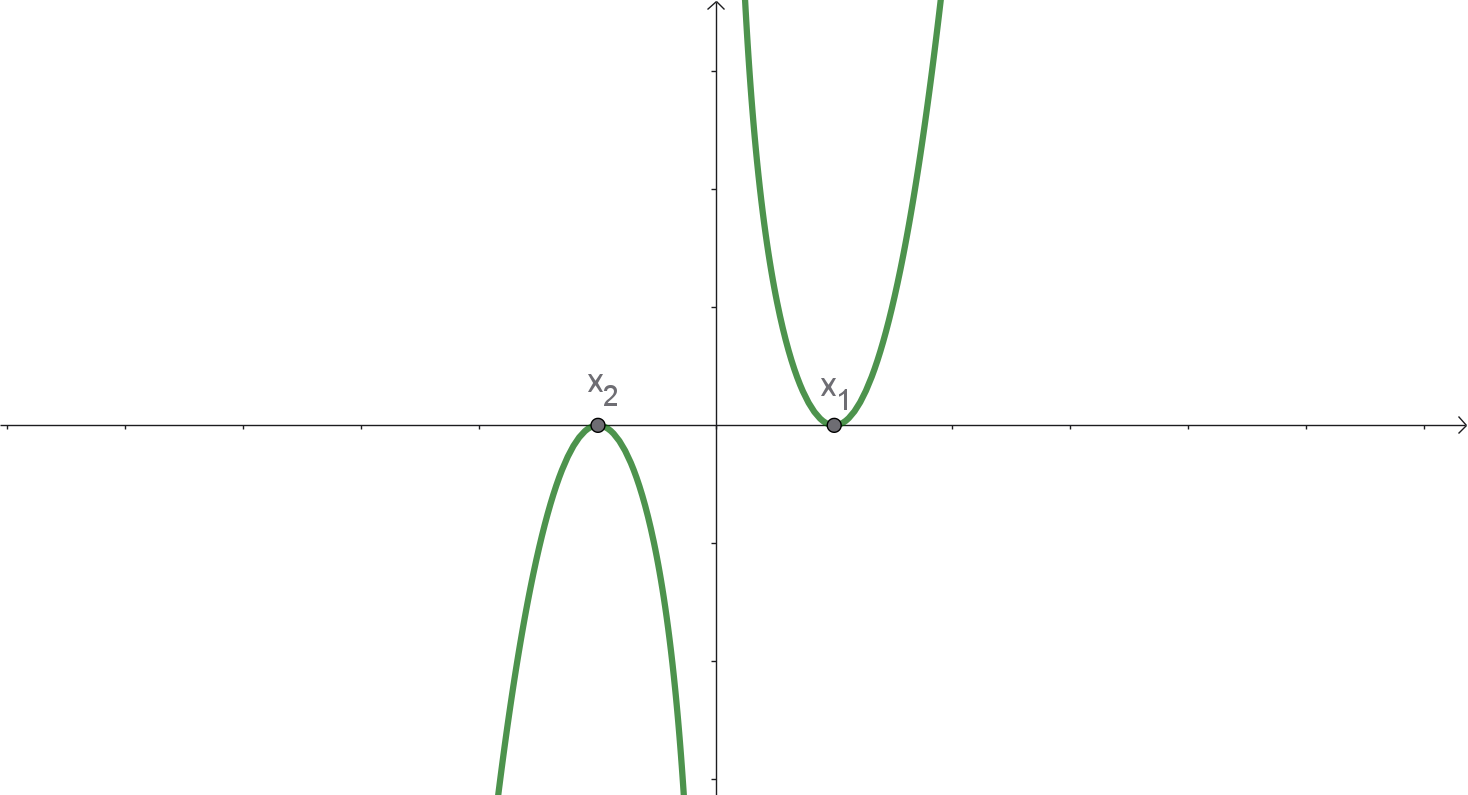}
	\caption{The degenerate case $G(x_1)=G(x_2)=0$.}
	\label{fig:G-degenerate}
\end{figure}

Now assume $a > 0$. For reference, the graph of $G(x)$ in this case is illustrated in Figure \ref{fig:G-nondegenerate}. If $\zeta = a$, then equation \eqref{eq fig} admits a double root at $x_2 < 0$ and two simple positive roots, yielding $g = 3$. By symmetry, if $\zeta = -a$, the system has a double root at $x_1 > 0$ and two simple negative roots, which again yields $g = 3$. For any $\zeta \in (-a, a)$, the line $y = \zeta/3$ lies strictly between the critical values, intersecting the graph of $G$ at four distinct simple roots, which corresponds to $g = 4$. This establishes the identification:
$$
f^{-1}(-a) \cup f^{-1}(a) = \{ p \in M \mid g(p) = 3 \}, \qquad f^{-1}(-a, a) = \Omega.
$$

\begin{figure}[!htbp]
	\centering
	\includegraphics[scale=0.8]{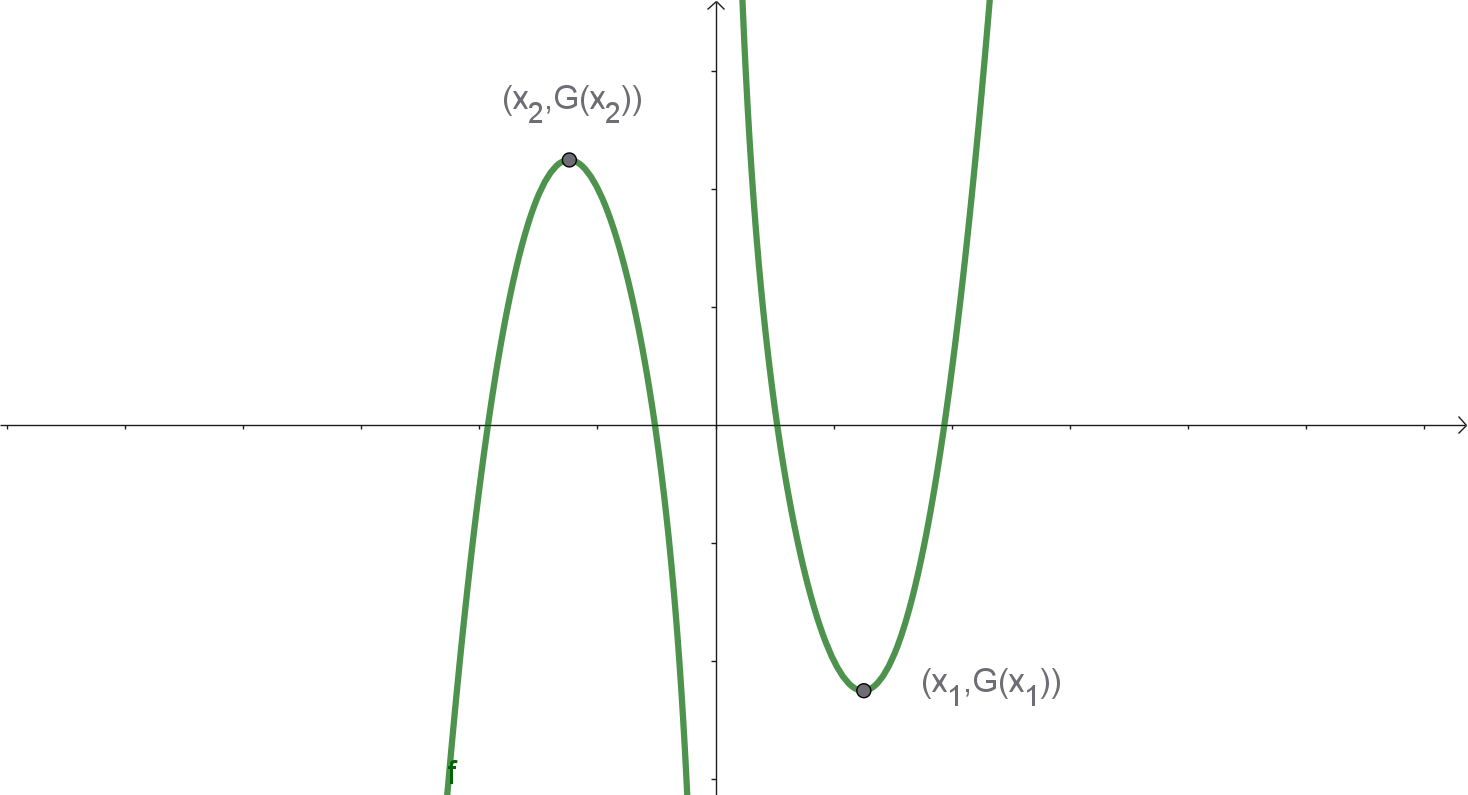}
	\caption{The nondegenerate case $G(x_1)<0<G(x_2)$.}
	\label{fig:G-nondegenerate}
\end{figure}

It remains to establish that $\Omega = \emptyset$ implies $a = 0$. Suppose, to the contrary, that $\Omega = \emptyset$ while $a > 0$. This hypothesis implies $f(M) \subset \{-a, a\}$. Since $M$ is connected and $f$ is smooth, $f$ must be constant on $M$, so $M$ is isoparametric. However, the root analysis above dictates that every point of $M$ would have exactly $g = 3$ distinct principal curvatures. By the classification, any minimal isoparametric hypersurface in $\mathbb{S}^{n+1}$ with $g = 3$ must have equal multiplicities of principal curvatures, requiring the dimension of the hypersurface to be a multiple of $3$, which contradicts $\dim M = 4$. Thus, the alternative $a > 0$ and $\Omega = \emptyset$ is geometrically obstructed, completing the proof.
			
		\end{proof}

For the remainder of this subsection, we assume $\Omega \neq \emptyset$, which implies $a > 0$. Following the variational cut-off framework established in \cite{ty23}, we partition $M$ into three disjoint strata:
\begin{align*}
	X &:= \{p \in M \mid f(p) = -a\} = f^{-1}(-a), \\
	Y &:= \{p \in M \mid -a < f(p) < a\} = \Omega, \\
	Z &:= \{p \in M \mid f(p) = a\} = f^{-1}(a).
\end{align*}
For any small parameter $0 < \epsilon < a$, we decompose $Y$ as $Y = X_\epsilon \cup Y_\epsilon \cup Z_\epsilon$, where
\begin{align*}
	X_{\epsilon} &:= \{p \in M \mid -a < f(p) < -a + \epsilon\}, \\
	Y_{\epsilon} &:= \{p \in M \mid -a + \epsilon \le f(p) \le a - \epsilon\}, \\
	Z_{\epsilon} &:= \{p \in M \mid a - \epsilon < f(p) < a\}.
\end{align*}
This decomposition ensures that the compact domain $Y_{\epsilon}$ remains a positive distance away from the multiple-root boundaries $X \cup Z$, while the boundary contributions generated on the collar regions $X_{\epsilon}$ and $Z_{\epsilon}$ can be analytically controlled as $\epsilon \to 0$.

\begin{prop}\label{propK>0}
			Let $M^4$ be a closed oriented minimal hypersurface immersed in $\mathbb{S}^5$ with constant $S$ and constant $K > 0$. Suppose that
			$\Omega\ne\emptyset$. If either
			\begin{equation}\label{eq:inequalities}
				6K - \frac{1}{2}S \ge 0 \qquad \text{or} \qquad 6K - \frac{1}{4}S^2 + \frac{5}{2}S \ge 0,
			\end{equation}
			then $M$ is isoparametric with $6K - \frac{1}{2}S=0$ or $6K - \frac{1}{4}S^2 + \frac{5}{2}S =0$, respectively.
\end{prop}

\begin{proof}
	We present the proof assuming the first inequality in \eqref{eq:inequalities}; the second case follows analogously by substituting the weighted form $\Psi_3$ in place of $\Psi_{-1}$.
	
	Let $\eta_{\epsilon} \in C^\infty(\mathbb{R}, [0, 1])$ be a smooth cut-off function satisfying:
	\begin{enumerate}
		\item $\eta_{\epsilon}(t) = 0$ on $(-\infty, -a + \frac{\epsilon}{4}] \cup [a - \frac{\epsilon}{4}, +\infty)$,
		\item $\eta_{\epsilon}(t) = 1$ on $[-a + \epsilon, a - \epsilon]$,
		\item $\eta_{\epsilon}'(t) \ge 0$ on $(-\infty, 0)$ and $\eta_{\epsilon}'(t) \le 0$ on $(0, +\infty)$.
	\end{enumerate}
	Since $\eta_{\epsilon} \circ f$ vanishes in a neighborhood of $M \setminus \Omega$, the form $(\eta_{\epsilon} \circ f) \Psi_{-1}$, initially defined only on $\Omega$, extends smoothly to the entire manifold $M$ by setting it to zero on $M \setminus \Omega$. Since $\eta_\epsilon \circ f$ has compact support in $Y$, Stokes' theorem yields
	\begin{equation}\label{eq stokes 1}
		\int_Y (\eta_{\epsilon} \circ f) d\Psi_{-1} + \int_Y (\eta_{\epsilon}' \circ f) df \wedge \Psi_{-1} = 0.
	\end{equation}

On the open set $Y$, the principal curvatures are pairwise distinct. Differentiating the constant invariants $H$, $S$, $P_4$ along with the variable $f$ yields the matrix system:
\begin{equation*}
	\begin{pmatrix}
		1 & 1 & 1 & 1 \\
		\lambda_1 & \lambda_2 & \lambda_3 & \lambda_4 \\
		\lambda_1^2 & \lambda_2^2 & \lambda_3^2 & \lambda_4^2 \\
		\lambda_1^3 & \lambda_2^3 & \lambda_3^3 & \lambda_4^3
	\end{pmatrix}
	\begin{pmatrix}
		h_{11j} \\ h_{22j} \\ h_{33j} \\ h_{44j}
	\end{pmatrix}
	=
	\begin{pmatrix}
		0 \\ 0 \\ e_j(f)/3 \\ 0
	\end{pmatrix}.
\end{equation*}
Since the principal curvatures are distinct on $Y$, the Vandermonde coefficient matrix is non-singular. Applying Cramer's rule yields the following expressions for the derivatives (where $f_j := e_j(f)$):
\begin{equation}\label{eq hiij solution}
	h_{iij} = -\frac{f_j}{3} \frac{\lambda_i}{\prod_{k \neq i}(\lambda_k - \lambda_i)}.
\end{equation}
Substituting \eqref{eq hiij solution} into the wedge products $df \wedge \Psi_{-1}$ and $df \wedge \Psi_{3}$, we obtain
	\begin{align}\label{eq df wedge psi}
	df\wedge \Psi_{-1}=&\big(\frac{\lambda_3\lambda_4}{\lambda_1-\lambda_2}f_1h_{221}-\frac{\lambda_3\lambda_4}{\lambda_1-\lambda_2}f_2h_{112}+\frac{\lambda_2\lambda_4}{\lambda_1-\lambda_3}f_1h_{331}-\frac{\lambda_2\lambda_4}{\lambda_1-\lambda_3}f_3h_{113}\\
	&+\frac{\lambda_2\lambda_3}{\lambda_1-\lambda_4}f_1h_{441}-\frac{\lambda_2\lambda_3}{\lambda_1-\lambda_4}f_4h_{114}+\frac{\lambda_1\lambda_4}{\lambda_2-\lambda_3}f_2h_{332}-\frac{\lambda_1\lambda_4}{\lambda_2-\lambda_3}f_3h_{223}\notag\\
	&+\frac{\lambda_1\lambda_3}{\lambda_2-\lambda_4}f_2h_{442}-\frac{\lambda_1\lambda_3}{\lambda_2-\lambda_4}f_4h_{224}+\frac{\lambda_1\lambda_2}{\lambda_3-\lambda_4}f_3h_{443}-\frac{\lambda_1\lambda_2}{\lambda_3-\lambda_4}f_4h_{334}\big)~\text{vol}\notag\\
	:=&\sum_{i=1}^4u_i f_i^2\,\mathrm{vol},\notag
\end{align}
\begin{align}\label{eq df wedge phi}
	df\wedge\Psi_{3}
	=&\big(\frac{\lambda_1^2+\lambda_2^2+\lambda_1\lambda_2}{\lambda_1-\lambda_2}f_1h_{221}-\frac{\lambda_1^2+\lambda_2^2+\lambda_1\lambda_2}{\lambda_1-\lambda_2}f_2h_{112}+\frac{\lambda_1^2+\lambda_3^2+\lambda_1\lambda_3}{\lambda_1-\lambda_3}f_1h_{331}\\
	&-\frac{\lambda_1^2+\lambda_3^2+\lambda_1\lambda_3}{\lambda_1-\lambda_3}f_3h_{113}+\frac{\lambda_1^2+\lambda_4^2+\lambda_1\lambda_4}{\lambda_1-\lambda_4}f_1h_{441}-\frac{\lambda_1^2+\lambda_4^2+\lambda_1\lambda_4}{\lambda_1-\lambda_4}f_4h_{114}\notag\\
	&+\frac{\lambda_2^2+\lambda_3^2+\lambda_2\lambda_3}{\lambda_2-\lambda_3}f_2h_{332}-\frac{\lambda_2^2+\lambda_3^2+\lambda_2\lambda_3}{\lambda_2-\lambda_3}f_3h_{223}+\frac{\lambda_2^2+\lambda_4^2+\lambda_2\lambda_4}{\lambda_2-\lambda_4}f_2h_{442}\notag\\
	&-\frac{\lambda_2^2+\lambda_4^2+\lambda_2\lambda_4}{\lambda_2-\lambda_4}f_4h_{224}+\frac{\lambda_3^2+\lambda_4^2+\lambda_3\lambda_4}{\lambda_3-\lambda_4}f_3h_{443}-\frac{\lambda_3^2+\lambda_4^2+\lambda_3\lambda_4}{\lambda_3-\lambda_4}f_4h_{334}\big)~\text{vol}\notag \\
	:=&\sum_{i=1}^4v_i f_i^2\,\mathrm{vol}.\notag
\end{align}
where $\{u_i\}$ and $\{v_i\}$ are smooth functions on $Y$. We shall employ the following boundary estimate, whose detailed proof is deferred to Subsection~\ref{3.5}:

\begin{asser}\label{asser}
	There exists a uniform constant $\mathcal{C} > 0$ depending only on $S$ and $K$ such that, for $K > 0$, we have $u_i, v_i \ge -\mathcal{C}$ on $Z_{\epsilon}$ and $u_i, v_i \le \mathcal{C}$ on $X_{\epsilon}$ for all $i \in \{1, 2, 3, 4\}$.
\end{asser}

We claim that the cut-off boundary contribution satisfies
\begin{equation}\label{eq cutoff-limit}
	\lim_{\epsilon \to 0} \int_Y (\eta_{\epsilon} \circ f) d\Psi_{-1} = 0.
\end{equation}
Indeed, if $X \cup Z = \emptyset$, then $Y = M$. For sufficiently small $\epsilon$, the support of $f(M)$ is contained in $[-a+\epsilon, a-\epsilon]$, which implies $\eta_{\epsilon} \circ f \equiv 1$ on $M$, and \eqref{eq cutoff-limit} follows directly from Stokes' theorem. 

Now assume $X \cup Z \neq \emptyset$. Under our sign conventions \eqref{eq:ci-negative}, the assumption $6K - \frac{1}{2}S \ge 0$ implies
$$
-d\Psi_{-1} = \left( \sum_{i=1}^4 (-c_i) h_{44i}^2 + 6K - \frac{1}{2}S \right) \operatorname{vol} \ge 0 \qquad \text{on } Y.
$$
Combining this with \eqref{eq stokes 1}, \eqref{eq df wedge psi}, and Assertion~\ref{asser}, we obtain the estimate:
\begin{equation}\label{eq cutoff-bound}
	\begin{aligned}
		0 \le \int_Y -(\eta_{\epsilon} \circ f) d\Psi_{-1} 
		&=\int_Y (\eta_\epsilon'\circ f)\,df\wedge\Psi_{1}= \int_{X_{\epsilon} \cup Z_{\epsilon}} (\eta_{\epsilon}' \circ f) \sum_{i=1}^4 u_i f_i^2 \operatorname{vol} \\
		&= \int_{X_{\epsilon}} |\eta_{\epsilon}' \circ f| \sum_{i=1}^4 u_i f_i^2 \operatorname{vol} + \int_{Z_{\epsilon}} |\eta_{\epsilon}' \circ f| \sum_{i=1}^4 (-u_i) f_i^2 \operatorname{vol} \\
		&\le \mathcal{C} \int_Y |\eta_{\epsilon}' \circ f| |df|^2 \operatorname{vol}.
	\end{aligned}
\end{equation}
We define the auxiliary smooth function $\xi_{\epsilon} \in C^\infty(\mathbb{R})$ by
\begin{equation*}
	\xi_{\epsilon}(t) = \begin{cases}
		\eta_{\epsilon}(t) - 1 & \text{on } (-\infty, 0], \\
		1 - \eta_{\epsilon}(t) & \text{on } [0, +\infty),
	\end{cases}
\end{equation*}
so that $\xi_{\epsilon}' = |\eta_{\epsilon}'|$. Let $\star$ denote the Hodge star operator. Since $M$ is closed, applying Stokes' theorem to $d\star \big((\xi_\epsilon \circ f)df\big) = (\xi_\epsilon' \circ f)|df|^2 \operatorname{vol} + (\xi_\epsilon \circ f)\Delta f \operatorname{vol}$ yields
\begin{equation}\label{eq endpoint-delta}
	\int_Y |\eta_\epsilon' \circ f| |df|^2 \operatorname{vol} = \int_M (\xi_\epsilon' \circ f)|df|^2 \operatorname{vol} = -\int_M (\xi_\epsilon \circ f)\Delta f \operatorname{vol}.
\end{equation}
Applying the dimension-independent endpoint analysis established in \cite[Section 4]{ty23}, we have
$$
\lim_{\epsilon \to 0} \int_{M \setminus Y_{\epsilon}} |\Delta f| \operatorname{vol} = 0.
$$
Since $|\xi_{\epsilon}| \le 1$ and $\operatorname{supp}(\xi_{\epsilon} \circ f) \subset M \setminus Y_{\epsilon}$, relation \eqref{eq endpoint-delta} guarantees that
$$
\int_Y |\eta_{\epsilon}' \circ f| |df|^2 \operatorname{vol} \le \int_{M \setminus Y_{\epsilon}} |\xi_{\epsilon} \circ f| |\Delta f| \operatorname{vol} \le \int_{M \setminus Y_{\epsilon}} |\Delta f| \operatorname{vol} \longrightarrow 0 \qquad \text{as } \epsilon \to 0.
$$
In view of \eqref{eq cutoff-bound}, this establishes the limit \eqref{eq cutoff-limit}.

Since $6K-\tfrac{1}{2}S \ge 0$, $-c_i > 0$ on $Y$, and $\eta_{\epsilon} \circ f = 1$ on $Y_{\epsilon}$, we have
$$
\int_Y -(\eta_{\epsilon} \circ f) d\Psi_{-1} \ge \int_{Y_{\epsilon}} \sum_{i=1}^4 (-c_i) h_{44i}^2 \operatorname{vol}.
$$
Taking the limit as $\epsilon \to 0$, the squeeze theorem yields
$$
\int_Y \sum_{i=1}^4 (-c_i) h_{44i}^2 \operatorname{vol} = 0,
$$
which forces $h_{44i} = 0$ on $Y$ for all $i \in \{1, 2, 3, 4\}$. By \eqref{eq hiij solution}, the vanishing of $\{h_{44i}\}$ implies $df = 0$ on $Y$. Because $M = X \cup Y \cup Z$, and the smooth function $f$ trivially satisfies $df = 0$ on the global extremum sets $X$ and $Z$, we conclude that $df \equiv 0$ on $M$. Thus, $f$ is constant, and $M$ is isoparametric.  

Moreover, \eqref{eq cutoff-limit} together with $\{h_{44i}=0\}$ implies $6K-\frac{1}{2}S=0$. Analogously, when it comes to $\Psi_{3}$, we have $\frac{1}{4}S^{2}-\frac{5}{2}S-6K=0$.

\end{proof}

\begin{rem}
	An analogous cut-off analysis applies to the negative curvature regime $K = \text{const.} < 0$ by utilizing the weighted form $\Psi_{-1}$ on the set $M \setminus \Omega$.
\end{rem}

To remove the auxiliary inequality constraints \eqref{eq:inequalities}, we recall the classical Gauss-Bonnet-Chern formula in our setting:
		
\begin{lem}[\cite{zhang01}]\label{GBC}
			Let $N^{2k}$ be a closed, even-dimensional Riemannian manifold. Then the Euler characteristic of $N$ is given by
			$$
			\chi(N) = \left( \frac{-1}{2\pi} \right)^k \int_N \operatorname{Pf}(R),
			$$
			where $\operatorname{Pf}(R)$ represents the Pfaffian of the curvature $2$-form matrix:
			\begin{equation*}
				\mathrm{Pf}(R):=\frac{1}{2^{k}k!}\sum_{i_{1},i_{2},\dots,i_{2k}}\delta_{i_{1},i_{2},\dots,i_{2k}}^{1,2,\dots ,2k}R_{i_{1}i_{2}}\wedge\cdots\wedge R_{i_{2k-1}i_{2k}},
			\end{equation*}
			where 
			\begin{equation*}
				\delta_{i_{1},i_{2},\dots,i_{2k}}^{1,2,\dots ,2k}=
				\left\{
				\begin{aligned}
					1&, \text{ if $(i_{1},i_{2},\dots,i_{2k})$ is an even permutation of $(1,2,\dots ,2k)$},\\
					-1&, \text{ if $(i_{1},i_{2},\dots,i_{2k})$ is an odd permutation of $(1,2,\dots ,2k)$},\\
					0&, \text{ else}.
				\end{aligned}
				\right.
			\end{equation*}
\end{lem}

In our context ($n=4$), the localized curvature forms are given by $R_{ij} = (1 + \lambda_i \lambda_j) \omega_i \wedge \omega_j$. A direct algebraic evaluation of the Pfaffian yields
		\begin{equation*}
			\begin{aligned}
				\operatorname{Pf}(R) &= R_{12} \wedge R_{34} - R_{13} \wedge R_{24} + R_{14} \wedge R_{23} \\
				&= \left( 3 + H_2 + 3K \right) \operatorname{vol} = \left( 3 - \frac{1}{2}S + 3K \right) \operatorname{vol},
			\end{aligned}
		\end{equation*}
		where we have used the minimality condition $H = 0$. Consequently, Lemma~\ref{GBC} leads to the integral identity:
\begin{equation}\label{eq gauss-bonnet k s}
			\int_M \left( 3 - \frac{1}{2}S + 3K \right) \operatorname{vol} = 4\pi^2 \chi(M).
\end{equation}

In Proposition~\ref{prop euler-vanishing} (established in the subsequent subsection), we prove that the Euler characteristic vanishes, i.e., $\chi(M) = 0$, whenever $K > 0$ and $\Omega \neq \emptyset$. Since both $S$ and $K$ are constant on $M$, equation \eqref{eq gauss-bonnet k s} yields the crucial identity:
\begin{equation}\label{eq:S-6K-relation}
	S = 6K + 6.
\end{equation}
Substituting \eqref{eq:S-6K-relation} into the auxiliary inequalities of Proposition~\ref{propK>0}, we find
$$
6K - \frac{1}{2}S = 3(K - 1),
$$
and
$$
6K - \frac{1}{4}S^2 + \frac{5}{2}S = -3(3K + 2)(K - 1).
$$
Thus, the first inequality in \eqref{eq:inequalities} is satisfied for all $K \ge 1$, while the second is satisfied for all $0 < K \le 1$. These ranges cover the entire positive regime $K > 0$. Furthermore, from the above two identities, both $6K - \frac{1}{2}S = 0$ and $6K - \frac{1}{4}S^2 + \frac{5}{2}S = 0$ yield $K=1$, which is equivalent to $S=12$. Consequently, the scalar curvature
$$
R_{M} = 12 - S \equiv 0
$$
coincides with the conclusion in \cite{ty23}. % that all such minimal isoparametric hypersurfaces have non-negative scalar curvature. 
Indeed, since $\Omega \neq \emptyset$, the classification of isoparametric hypersurfaces implies that $M$ must be the isoparametric hypersurface with $g=4$ simple principal curvatures. This completes the proof of Theorem~\ref{thm:main}.

\hfill $\Box$

%-------------------------------------------------------------------------------------------

\subsection{Proof of $\chi(M)=0$.}\label{3.4}

In this subsection, we establish the vanishing of the Euler characteristic $\chi(M)$ for any closed minimal hypersurface $M^4 \subset \mathbb{S}^5$ with constant $S$ and constant $K > 0$, under the assumption that $\Omega \neq \emptyset$. This topological result completes the variational rigidity proof presented in Subsection~\ref{3.3}.

We begin by recalling the topological relationship between orientability, the first Stiefel-Whitney class, and the orientation double cover (we refer the reader to \cite{gtm82,cc, zhang01} for differential-topological details). Let $\pi: E \to M$ be a real vector bundle of rank $r$. The orientation bundle of $E$, denoted by $\widetilde{\pi}: \mathcal{O}(E) \to M$, is a canonical principal $\mathbb{Z}_2$-bundle over $M$ whose fiber $\mathcal{O}_p(E) := \widetilde{\pi}^{-1}(p)$ at each point $p \in M$ consists of the two possible orientations of the fiber $E_p$. The total space $\mathcal{O}(E)$ forms a $2$-sheeted covering space of $M$, and $\mathcal{O}(E)$ is trivial if and only if $E$ is orientable. 
In the language of characteristic classes, the obstruction to orientability is measured by the first Stiefel-Whitney class $w_1(E) \in H^1(M; \mathbb{Z}_2)$, which vanishes if and only if $E$ is orientable. By construction, the pullback bundle over the total space of its own orientation cover, $\widetilde{\pi}^* E \to \mathcal{O}(E)$, admits a canonical, continuous orientation section induced by the diagonal map, rendering $\widetilde{\pi}^* E$ globally orientable.

\medskip
\noindent\textbf{Reduction to the Oriented Case.}
If the manifold $M$ itself is non-orientable, we let $\pi_M: \widetilde{M} := \mathcal{O}(TM) \to M$ be its orientation double cover. Equipped with the pullback metric and the pullback minimal immersion, $\widetilde{M}$ is a closed, oriented minimal hypersurface in $\mathbb{S}^5$. Since the geometric quantities $S$, $K$, and $f$ on $\widetilde{M}$ are the pullbacks of those on $M$, the geometric constants $S$ and $K$ remain invariant, and the condition $\Omega \neq \emptyset$ is preserved. Because the Euler characteristic behaves multiplicatively under finite covering maps, we have $\chi(\widetilde{M}) = 2\chi(M)$. Consequently, showing $\chi(\widetilde{M}) = 0$ immediately implies $\chi(M) = 0$. Thus, without loss of generality, we assume $M$ is oriented throughout the remainder of the proof.

\medskip
Now, let $E \to M$ be an oriented real vector bundle of rank $2k$ equipped with a fiber metric $g^E$ and a compatible connection $\nabla^E$. In terms of a local orthonormal frame, the connection forms $\omega^E$ and curvature forms $R^E$ satisfy the structural equation 
\begin{equation}\label{eq str v.b.}
	d\omega^E = \omega^E \wedge \omega^E - R^E.
\end{equation}
The Euler form of $(E, \nabla^E)$ is defined by
\begin{equation}\label{euler form}
	e(E, \nabla^E) := \left(\frac{-1}{2\pi}\right)^k \operatorname{Pf}(R^E),
\end{equation}
which is a closed $2k$-form whose de Rham cohomology class $e(E) = [e(E, \nabla^E)] \in H^{2k}_{\mathrm{dR}}(M)$ is the Euler class of $E$, which is independent of the choice of $g^E$ and $\nabla^E$. In the special case where $\operatorname{rank} E = 2$, both the connection form and the curvature form are scalar-valued, reducing the relation \eqref{eq str v.b.} and the Euler form \eqref{euler form} to
\begin{equation*}
	R^E = -d\omega^E, \qquad e(E, \nabla^E) = -\frac{1}{2\pi} R^E = \frac{1}{2\pi} d\omega^E.
\end{equation*}

\begin{prop}\label{prop euler-vanishing}
	Let $M^4$ be a closed minimal hypersurface immersed in $\mathbb{S}^5$ with constant $S$ and constant $K > 0$. If $\Omega \neq \emptyset$, then $\chi(M) = 0$.
\end{prop}

\begin{proof}
By the reduction argument above, we assume throughout that $M$ is oriented.

	First, we consider the case where $\operatorname{Im}f \subsetneq [-a, a]$.  In this case, at least one of the boundary values $\pm a$ is not achieved on $M$.

	Suppose first that $a \notin \operatorname{Im} f$. By the algebraic characterization in Lemma~\ref{lem lem3.3 figure}, the minimum principal curvature $\lambda_1$ must have multiplicity one globally on $M$. The corresponding eigenspaces then form a smooth real line bundle $L_1 \subset TM$. If $L_1$ is orientable, it admits a globally defined, nowhere-vanishing unit section. By the classical Poincaré-Hopf theorem, this implies $\chi(M) = 0$. If $L_1$ is non-orientable, let $\pi: \mathcal{O}(L_1) \to M$ be its orientation double cover. The pullback line bundle $\pi^* L_1$ is orientable, and thus carries a global unit section. Since $\pi^* L_1 \subset \pi^* TM \cong T(\mathcal{O}(L_1))$, this section yields a nowhere-vanishing vector field on the closed manifold $\mathcal{O}(L_1)$, forcing $\chi(\mathcal{O}(L_1)) = 0$. The covering formula $\chi(\mathcal{O}(L_1)) = 2\chi(M)$ then guarantees $\chi(M) = 0$. An identical argument holds if $-a \notin \operatorname{Im} f$ by considering the simple line bundle associated with $\lambda_4$.
	
	For the remainder of the proof, we assume $\operatorname{Im} f = [-a, a]$. We define an open cover of $M$ by
	$$
	\Omega_{12} := f^{-1}\left(-\infty, \frac{a}{2}\right), \qquad \Omega_{34} := f^{-1}\left(-\frac{a}{2}, +\infty\right).
	$$
	By Lemma~\ref{lem lem3.3 figure}, the eigenvalues $\lambda_1, \lambda_2$ are simple on $\Omega_{12}$, and $\lambda_3, \lambda_4$ are simple on $\Omega_{34}$. Thus, the corresponding eigenspaces define smooth line bundles $L_1, L_2$ over $\Omega_{12}$ and $L_3, L_4$ over $\Omega_{34}$. We define two global rank-$2$ subbundles $E_{12}, E_{34} \subset TM$ by
	\begin{equation*}
		E_{12} := \begin{cases}
			L_1 \oplus L_2 & \text{on } \Omega_{12}, \\
			(L_3 \oplus L_4)^\perp & \text{on } \Omega_{34},
		\end{cases} \qquad
		E_{34} := \begin{cases}
			(L_1 \oplus L_2)^\perp & \text{on } \Omega_{12}, \\
			L_3 \oplus L_4 & \text{on } \Omega_{34}.
		\end{cases}
	\end{equation*}
	This construction is smoothly well-defined on $M$. On the intersection $\Omega_{12} \cap \Omega_{34} \subset \Omega$, all four principal curvatures are simple, which matches the identity $(L_3 \oplus L_4)^\perp = L_1 \oplus L_2$. This yields the global topological Whitney sum splitting:
	\begin{equation}\label{eq:TM-splitting}
		TM = E_{12} \oplus E_{34}.
	\end{equation}

Since $M$ is oriented, its first Stiefel-Whitney class vanishes, i.e., $w_1(TM) = 0$. Applying the Whitney product formula to the splitting \eqref{eq:TM-splitting}, we obtain
$$
0 = w_1(TM) = w_1(E_{12}) + w_1(E_{34}).
$$
This algebraic identity implies $w_1(E_{12}) = w_1(E_{34}) \in H^1(M; \mathbb{Z}_2)$. In terms of double covers, this topologically guarantees that the respective orientation covers of $E_{12}$ and $E_{34}$ are isomorphic. We can therefore define a unique, common orientation double cover space for both subbundles, denoted by the projection map
\begin{equation}\label{eq:common-cover}
	\widetilde{\pi}: \widetilde{M} := \mathcal{O}(E_{12}) \cong \mathcal{O}(E_{34}) \longrightarrow M.
\end{equation}
In particular, $E_{12}$ is orientable if and only if $E_{34}$ is orientable.

We first analyze the case where $E_{12}$ and $E_{34}$ are orientable. The Whitney product formula for the Euler class yields
$$
e(TM) = e(E_{12}) \wedge e(E_{34}).
$$
We shall represent these Euler classes by differential forms with disjoint supports. Let $g^{E_{12}}$ be the induced metric on $E_{12}$ and $\nabla^{E_{12}}$ be a compatible connection. On the open set
$$
\widetilde{\Omega}_{12} := f^{-1}\left(-\infty, \frac{3a}{4}\right),
$$
the eigenvalues $\lambda_1, \lambda_2$ remain simple. We locally choose oriented unit principal fields $\{e_1, e_2\}$ spanning $L_1, L_2$, compatible with the orientation of $E_{12}$. Any other oriented principal frame must satisfy $(e_1, e_2) \mapsto (-e_1, -e_2)$, which implies that the connection $1$-form
$$
\omega^{E_{12}} := \langle \nabla^{E_{12}} e_1, e_2 \rangle_{g^{E_{12}}}
$$
is globally well-defined on $\widetilde{\Omega}_{12}$. On this domain, the curvature form satisfies $R^{E_{12}} = -d\omega^{E_{12}}$.

Now, choose a smooth cut-off function $\phi \in C^\infty(\mathbb{R}, [0, 1])$ such that 
$$
\phi(t) = 1 \quad \text{for } t < \frac{a}{2}, \qquad \phi(t) = 0 \quad \text{for } t > \frac{3a}{4}.
$$
Extending by zero, the form $(\phi \circ f)\omega^{E_{12}}$ is globally well-defined on $M$. We define the modified Euler representative $2$-form on $M$ by
$$
\mathcal{E}_{12} := -\frac{1}{2\pi} R^{E_{12}} - \frac{1}{2\pi} d\big((\phi \circ f)\omega^{E_{12}}\big).
$$
Since $\mathcal{E}_{12}$ differs from the standard Euler form by an exact form, it represents the Euler class $e(E_{12})$. Furthermore, on $\Omega_{12}$, we have $\phi \circ f = 1$, which yields
$$
\mathcal{E}_{12}\big|_{\Omega_{12}} = -\frac{1}{2\pi} R^{E_{12}} - \frac{1}{2\pi} d\omega^{E_{12}} = 0.
$$

An identical construction applies to the subbundle $E_{34}$. On the domain $\widetilde{\Omega}_{34} := f^{-1}(-\frac{3a}{4}, +\infty)$, the eigenvalues $\lambda_3, \lambda_4$ remain simple. Choosing a smooth cut-off function $\psi \in C^\infty(\mathbb{R}, [0, 1])$ with $\psi(t) = 1$ for $t > -\frac{a}{2}$ and $\psi(t) = 0$ for $t < -\frac{3a}{4}$, we obtain a closed representative $2$-form $\mathcal{E}_{34} \in e(E_{34})$ satisfying
$$
\mathcal{E}_{34}\big|_{\Omega_{34}} = 0.
$$
Since $\Omega_{12} \cup \Omega_{34} = M$, the pointwise wedge product of these representatives vanishes identically on $M$:
$$
\mathcal{E}_{12} \wedge \mathcal{E}_{34} \equiv 0.
$$
Thus, the Euler class of $TM$ vanishes:
$$
e(TM) = [\mathcal{E}_{12} \wedge \mathcal{E}_{34}] = 0 \in H^4_{\mathrm{dR}}(M),
$$
which immediately yields $\chi(M) = 0$.

It remains to address the case where the subbundles $E_{12}$ and $E_{34}$ are non-orientable. Using the projection map $\widetilde{\pi}: \widetilde{M} \to M$ of the common orientation double cover defined in \eqref{eq:common-cover}, the pulled-back bundles $\widetilde{\pi}^* E_{12}$ and $\widetilde{\pi}^* E_{34}$ are orientable over the oriented closed manifold $\widetilde{M}$. Via the differential of the covering map, we obtain the smooth splitting
\begin{align*}
	T\widetilde{M} &= \widetilde{\pi}^* E_{12} \oplus \widetilde{\pi}^* E_{34}, \\
	\widetilde{\pi}^* E_{12} &= \widetilde{\pi}^* L_1 \oplus \widetilde{\pi}^* L_2 \quad \text{on } \widetilde{\pi}^{-1}(\Omega_{12}), \\
	\widetilde{\pi}^* E_{34} &= \widetilde{\pi}^* L_3 \oplus \widetilde{\pi}^* L_4 \quad \text{on } \widetilde{\pi}^{-1}(\Omega_{34}).
\end{align*}
Applying the argument for the orientable case to the oriented closed manifold $\widetilde{M}$ yields $\chi(\widetilde{M}) = 0$. Using the double-cover relation $\chi(\widetilde{M}) = 2\chi(M)$, we conclude $\chi(M) = 0$.

\end{proof}

\begin{rem}
	In settings where a principal curvature, say $\lambda_1$, is simple everywhere on $M$, the corresponding eigenspace associated with $\lambda_1$ defines a global real line bundle $L_1 \subset TM$. This line bundle $L_1$ admits a global, nowhere-vanishing section if and only if it is orientable. If $L_1$ is non-orientable, one recovers the vanishing of the Euler characteristic by passing to its orientation double cover $\mathcal{O}(L_1)$. This clarifies a key orientability subtlety in the argument of \cite{cui25}.
\end{rem}

%-------------------------------------------------------------------------------------------

\vspace{3mm}

%------------------------------------------------------------------------------------
\subsection{Proof of Assertion \ref{asser}}\label{3.5}

	We first make precise the dependence of the principal curvatures on $f$. On
	$\Omega=f^{-1}(-a,a)$, the principal curvatures are the four ordered real
	roots of
	$$
	x^4-\frac S2x^2-\frac f3x+K=0 .
	$$
	Since $S$ and $K$ are fixed and the four roots are distinct for
	$-a<f<a$, each $\lambda_i$ is a smooth function of the single variable $f$ on
	$(-a,a)$. Moreover the ordered roots extend continuously to the endpoints.
	Thus, as $f\to a^-$ and $f\to -a^+$,
	$$
	(\lambda_1,\lambda_2,\lambda_3,\lambda_4)
	\longrightarrow
	(\beta_1,\beta_2,\beta_3,\beta_4),
	$$
	and
	$$
	(\lambda_1,\lambda_2,\lambda_3,\lambda_4)
	\longrightarrow
	(\alpha_1,\alpha_2,\alpha_3,\alpha_4),
	$$
	respectively, where
	$$
	\alpha_1<\alpha_2<0<\alpha_3=\alpha_4,
	\qquad
	\beta_1=\beta_2<0<\beta_3<\beta_4 .
	$$
	Consequently the coefficients $u_i,v_i$, being rational functions of the
	$\lambda_j$ on $\Omega$, are functions of $f$ alone.
	
	We prove the lower bound for $u_i,v_i$ on $Z_\epsilon$; the upper bound on
	$X_\epsilon$ is obtained by the same argument at the other endpoint
	$f=-a$, where the coalescing pair is $\lambda_3,\lambda_4$.
	
	Near $f=a$, the only possible singularity is caused by
	$\lambda_1-\lambda_2\to0$. From the formula for $df\wedge\Psi_1$, after using
	(3.24), one obtains
	$$
	\begin{aligned}
		u_1={}&
		-\frac{K}{3\lambda_1(\lambda_2-\lambda_1)^2(\lambda_4-\lambda_2)(\lambda_3-\lambda_2)}
		+\frac{K}{3\lambda_1(\lambda_3-\lambda_1)^2(\lambda_4-\lambda_3)(\lambda_3-\lambda_2)}\\
		&-\frac{K}{3\lambda_1(\lambda_4-\lambda_1)^2(\lambda_2-\lambda_4)(\lambda_3-\lambda_4)},
	\end{aligned}
	$$
	and
	$$
	\begin{aligned}
		u_2={}&
		-\frac{K}{3\lambda_2(\lambda_2-\lambda_1)^2(\lambda_3-\lambda_1)(\lambda_4-\lambda_1)}
		+\frac{K}{3\lambda_2(\lambda_3-\lambda_2)^2(\lambda_4-\lambda_3)(\lambda_3-\lambda_1)}\\
		&-\frac{K}{3\lambda_2(\lambda_4-\lambda_2)^2(\lambda_3-\lambda_4)(\lambda_4-\lambda_1)} .
	\end{aligned}
	$$
	Since $\beta_1=\beta_2<0<\beta_3<\beta_4$, the first terms in $u_1$ and
	$u_2$ tend to $+\infty$, while the remaining terms have finite limits. Hence
	$$
	u_1\to+\infty,\qquad u_2\to+\infty
	\qquad\text{as } f\to a^- .
	$$
	For $u_3$ and $u_4$, we have
	$$
	\begin{aligned}
		u_3={}&
		-\frac{K}{3\lambda_3(\lambda_3-\lambda_1)^2(\lambda_2-\lambda_1)(\lambda_4-\lambda_1)}
		+\frac{K}{3\lambda_3(\lambda_3-\lambda_2)^2(\lambda_4-\lambda_2)(\lambda_2-\lambda_1)}\\
		&-\frac{K}{3\lambda_3(\lambda_4-\lambda_3)^2(\lambda_1-\lambda_4)(\lambda_2-\lambda_4)},
	\end{aligned}
	$$
	and
	$$
	\begin{aligned}
		u_4={}&
		-\frac{K}{3\lambda_4(\lambda_4-\lambda_1)^2(\lambda_3-\lambda_1)(\lambda_2-\lambda_1)}
		+\frac{K}{3\lambda_4(\lambda_4-\lambda_2)^2(\lambda_3-\lambda_2)(\lambda_2-\lambda_1)}\\
		&-\frac{K}{3\lambda_4(\lambda_4-\lambda_3)^2(\lambda_1-\lambda_3)(\lambda_2-\lambda_3)} .
	\end{aligned}
	$$
	For $u_3$, write the three summands above as $u_{31},u_{32},u_{34}$. The term
	$u_{34}$ has a finite limit as $f\to a^-$.
	The potentially singular sum can be written as the difference quotient
	$$
	u_{31}+u_{32}
	=
	\frac{K}{3\lambda_3(\lambda_2-\lambda_1)}
	\left[
	\frac{1}{(\lambda_3-\lambda_2)^2(\lambda_4-\lambda_2)}
	-
	\frac{1}{(\lambda_3-\lambda_1)^2(\lambda_4-\lambda_1)}
	\right].
	$$
	Since $\lambda_3,\lambda_4$ stay away from $\lambda_1,\lambda_2$ near
	$f=a$, the function
	$$
	(t,z,w)\longmapsto
	\frac{1}{(z-t)^2(w-t)} ,
	$$
	with $z=\lambda_3$ and $w=\lambda_4$, is smooth in a neighborhood of
	$(\beta_1,\beta_3,\beta_4)$. Hence the above difference quotient is bounded,
	and in fact has a finite limit as $f\to a^-$. Thus $u_3$ has a finite limit
	as $f\to a^-$. The same argument gives a finite limit for $u_4$.

	The coefficients $v_i$ are treated in the same way. The formulas for $v_1$ and
	$v_2$ show that their singular first terms tend to $+\infty$, and hence
	$$
	v_1\to+\infty,\qquad v_2\to+\infty
	\qquad\text{as } f\to a^- .
	$$
	For $v_3$, denote by $v_{31},v_{32},v_{34}$ the three terms displayed in the
	following expression:
	$$
	\begin{aligned}
		v_3={}&
		\frac{-\lambda_1(\lambda_1^2+\lambda_1\lambda_3+\lambda_3^2)}
		{3(\lambda_3-\lambda_1)^2(\lambda_2-\lambda_1)(\lambda_4-\lambda_1)}-\frac{\lambda_2(\lambda_2^2+\lambda_2\lambda_3+\lambda_3^2)}
		{3(\lambda_3-\lambda_2)^2(\lambda_2-\lambda_1)(\lambda_4-\lambda_2)}\\
		&+\frac{\lambda_4(\lambda_3^2+\lambda_3\lambda_4+\lambda_4^2)}
		{3(\lambda_4-\lambda_3)^2(\lambda_4-\lambda_1)(\lambda_4-\lambda_2)} .
	\end{aligned}
	$$
	The term $v_{34}$ is finite, while
	$$
	\begin{aligned}
		v_{31}+v_{32}
		={}&
		\frac{1}{3(\lambda_2-\lambda_1)}
		\left[
		\frac{\lambda_2(\lambda_2^2+\lambda_2\lambda_3+\lambda_3^2)}
		{(\lambda_3-\lambda_2)^2(\lambda_4-\lambda_2)}
		-
		\frac{\lambda_1(\lambda_1^2+\lambda_1\lambda_3+\lambda_3^2)}
		{(\lambda_3-\lambda_1)^2(\lambda_4-\lambda_1)}
		\right].
	\end{aligned}
	$$
	This is again a difference quotient of the smooth function
	$$
	F(t,z,w):=\frac{t(t^2+tz+z^2)}{(z-t)^2(w-t)},
	$$
	which is well-defined and smooth in a neighborhood of
	$(\beta_1,\beta_3,\beta_4)$. Indeed, we have
	$$
	v_{31}+v_{32}
	=
	\frac{1}{3}
	\frac{F(\lambda_2,\lambda_3,\lambda_4)
		-F(\lambda_1,\lambda_3,\lambda_4)}
	{\lambda_2-\lambda_1}.
	$$
	Since
	$$
	(\lambda_1,\lambda_2,\lambda_3,\lambda_4)
	\to
	(\beta_1,\beta_1,\beta_3,\beta_4)
	\quad\text{as } f\to a^-,
	$$
	it follows that
	$$
	v_{31}+v_{32}
	\to
	\frac{1}{3}\partial_tF(\beta_1,\beta_3,\beta_4).
	$$
	Thus $v_{31}+v_{32}$ remains finite as $f\to a^-$. Since $v_{34}$ also has a
	finite limit, $v_3$ is finite at the endpoint. The proof for $v_4$ is
	analogous.

	It follows that each $u_i$ and $v_i$ has either a finite limit
	or the limit $+\infty$ as $f\to a^-$. Hence there exist $\delta>0$ and $\mathcal{C}_1>0$, depending
	only on $S$ and $K$, such that
	$$
	u_i\ge -\mathcal{C}_1,\qquad v_i\ge -\mathcal{C}_1
	\qquad\text{for } a-\delta<f<a .
	$$
	On the compact interval $[0,a-\delta]$, the functions $u_i,v_i$ are continuous
	and therefore bounded from below. Enlarging $\mathcal{C}_1$ if necessary, we obtain
	$$
	u_i\ge -\mathcal{C}_1,\qquad v_i\ge -\mathcal{C}_1
	\qquad\text{for } 0<f<a .
	$$
	Since $Z_\epsilon\subset\{0<f<a\}$, this gives the desired lower bound on
	$Z_\epsilon$.
	
	Repeating the same argument near $f=-a$, where
	$\lambda_3-\lambda_4\to0$, shows that each $u_i$ and $v_i$ has either a finite
	limit or the limit $-\infty$ as $f\to -a^+$. Therefore there exists
	$\mathcal{C}_2>0$, depending only on $S$ and $K$, such that
	$$
	u_i\le \mathcal{C}_2,\qquad v_i\le \mathcal{C}_2
	\quad\text{for } -a<f<0 .
	$$
	Since $X_\epsilon\subset\{-a<f<0\}$, the required upper bound on
	$X_\epsilon$ follows. 
	
	Taking $\mathcal{C}=\max\{\mathcal{C}_1,\mathcal{C}_2\}$ proves the Assertion \ref{asser}.

\hfill$\Box$

\begin{ack}
The authors would like to sincerely thank Professor Zizhou Tang for his helpful comments on the Euler characteristic.
\end{ack}

%------------------------------------------------------------------------------------

\end{document}